\documentclass[12pt]{article}

\usepackage{amsfonts}
\usepackage{amsmath,amssymb,amsthm}
\usepackage[dvips]{graphicx}
\usepackage{psfrag}

\newtheorem{theorem}{Theorem}
\newtheorem{corollary}[theorem]{Corollary}
\newtheorem{definition}[theorem]{Definition}
\newtheorem{lemma}[theorem]{Lemma}

\newtheorem{remark}[theorem]{Remark}

\newcommand{\RR}{\mathbb{R}}
\newcommand{\ZZ}{\mathbb{Z}}
\newcommand{\NN}{\mathbb{N}}

\newcommand{\dive}{\operatorname{div}}
\newcommand{\dps}{\displaystyle}
\newcommand{\eps}{\varepsilon}

\begin{document}

\title{An MsFEM type approach for \\ perforated domains}
\author{Claude Le Bris$^1$, Fr\'ed\'eric Legoll$^1$, Alexei Lozinski$^{2}$\\
{\footnotesize $^1$ \'Ecole Nationale des Ponts et
Chauss\'ees,}\\
{\footnotesize 6 et 8 avenue Blaise Pascal, 77455 Marne-La-Vall\'ee
Cedex 2, FRANCE}\\
{\footnotesize and}\\
{\footnotesize  INRIA Rocquencourt, MICMAC project-team,}\\
{\footnotesize  78153 Le Chesnay Cedex, FRANCE}\\
{\footnotesize\tt lebris@cermics.enpc.fr, legoll@lami.enpc.fr}\\
{\footnotesize $^2$ Formerly at Institut de Math\'ematiques de Toulouse,}\\
{\footnotesize Universit\'e Paul Sabatier,}\\
{\footnotesize 118 route de Narbonne, 31062 Toulouse Cedex 9, FRANCE}\\
{\footnotesize Now at Laboratoire de Math\'ematiques CNRS UMR 6623,}\\
{\footnotesize Universit\'e de Franche-Comt\'e, 16 route de Gray, 25030
  Besan\c con Cedex, FRANCE}\\
{\footnotesize \tt alexei.lozinski@univ-fcomte.fr} 
}

\maketitle

\begin{abstract}
 We follow up on our previous work~\cite{companion-article} where we
 have studied a multiscale finite element (MsFEM) type method in
 the vein of the classical Crouzeix-Raviart finite element method that is
 specifically adapted for highly oscillatory elliptic problems. We adapt
 the approach  to address here a multiscale problem on a perforated
 domain. An additional ingredient of our approach is the enrichment of
 the multiscale
 finite element space using bubble functions.  We first establish a
 theoretical  error
 estimate. We next show that, on the
 problem we consider, the approach we propose outperforms all
 dedicated existing variants of MsFEM we are aware of.
\end{abstract}


\section{Introduction}
\label{sec:introduction}

\subsection{Generalities}
\label{ssec:Generalities}

We consider a bounded domain~$\Omega \subset \RR^d$ and a set $B_\varepsilon$ of
perforations within this domain. The perforations are supposedly small
and in extremely large a number. The parameter~$\varepsilon$ stands here
for a typical distance between the perforations. We denote
by~$\Omega_\varepsilon=\Omega \setminus \overline{B_\varepsilon}$ the perforated
domain (see Figure~\ref{fig:perforation}). We then consider the following problem: find $u :
\Omega_\varepsilon \rightarrow \mathbb{R}$, solution of
\begin{equation}
\label{eq:genP}
-\Delta u = f \text{ in $\Omega_\varepsilon$},
\quad
u = 0\text{ on $\partial \Omega_\varepsilon$},
\end{equation}
where $f:\Omega \rightarrow \mathbb{R}$ is a given function, assumed
sufficiently regular on $\Omega$. It is important to note
that the homogeneous Dirichlet boundary condition on $\partial \Omega_\varepsilon$ (and hence on the boundary $\Omega \cap \partial
B_\varepsilon$ of the perforations) is a crucial feature of the problem we
consider. Our academic enterprise is motivated by various physically relevant
problems, for instance in fluid mechanics, atmospheric modeling, electrostatic devices, \dots A different boundary condition, such as a 
Neumann boundary condition, would lead to completely different
theoretical considerations and, eventually, a different numerical
approach.
The consideration of~\eqref{eq:genP} can also be seen as a
step toward the resolution of the Stokes problem on perforated
domains. In that latter case, homogeneous  Dirichlet boundary conditions on the
perforations are typical for many applicative contexts.

\medskip

\begin{figure}[htbp]
\psfrag{perf}{Perforations $B_\eps$}
\psfrag{bp}{Boundary $\partial B_\eps$ of the perforations}
\psfrag{dom}{Domain $\Omega_\eps$}
\centerline{
\includegraphics[width=7truecm]{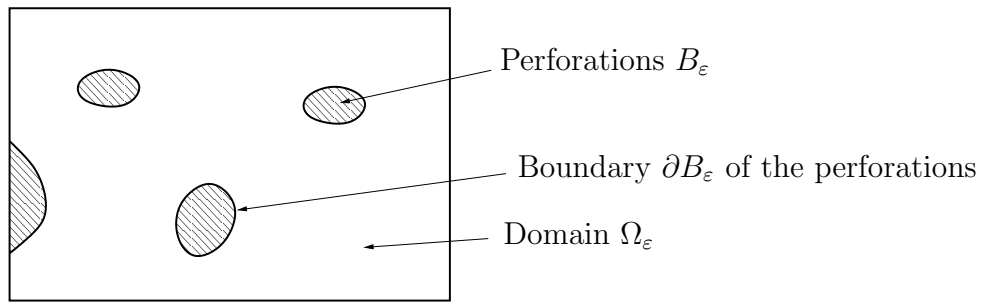}
}
\caption{The domain $\Omega$ contains perforations $B_\varepsilon$. The perforated domain is $\Omega_\varepsilon=\Omega \setminus \overline{B_\varepsilon}$. The boundary of $\Omega_\varepsilon$ is the union of $\partial B_\varepsilon \cap \overline{\Omega_\varepsilon}$ (the part of the boundary of the perforations that is included in $\overline{\Omega_\varepsilon}$) and of $\partial \Omega \cap \overline{\Omega_\varepsilon}$.
\label{fig:perforation}}
\end{figure}
 
\medskip

Our purpose here is to propose and study a dedicated multiscale finite element
method (MsFEM). To this end, we consider the variant
of MsFEM using \emph{Crouzeix-Raviart type} finite
elements~\cite{Crouzeix-Raviart} which we have employed 
and studied for a prototypical multiscale elliptic problem
in~\cite{companion-article} and we adapt the approach for the particular
setting under consideration here. The major adaptation we perform (and
thus one of the added values with respect to our earlier
work~\cite{companion-article}) is the addition of \emph{bubble
  functions} to the finite element basis set. 
Let us briefly comment upon the motivation for these two ingredients:
Crouzeix-Raviart type finite element on the one hand, and addition of
bubble functions on the other hand.
 
The motivation for using
Crouzeix-Raviart type finite elements stems from our wish to devise a
numerical approach as accurate as possible for a limited
computational workload. In general, it is well known that, for the construction of
multiscale finite elements, boundary conditions set on the edges
(facets) of mesh elements for the definition of the basis functions play
a critical role for the eventual accuracy and efficiency of the
approach. Using Crouzeix-Raviart type elements
(see~\cite{Crouzeix-Raviart} for their original introduction) gives a
definite flexibility. In short, the continuity of our multiscale finite
element basis set functions accross the edges of the mesh is enforced
only in a weak sense by requiring that the average of the jump
vanishes on each edge. This ``weak'' continuity condition leads to some 
natural boundary conditions for the multiscale basis functions (see Section~\ref{ssec:msfem}). The
nonconforming approximation obtained in this manner proves to be very
effective, see~\cite{companion-article}.
The above issue regarding boundary conditions on the mesh elements is all the more crucial when dealing with perforated
computational domains. Indeed,  we want the approach we construct to
be as insensitive as
possible  to the possible intersections between element edges
and perforations. The long term motivation for this is the wish to
address problems where the perforations can be very heterogeneously
distributed (think, say, of nonperiodic, or even
random arrays of perforations). The {\it ad hoc} construction of a mesh that
(essentially) avoids intersecting the perforations is then prohibitively
difficult.  

The second ingredient of our approach is the addition of bubble
functions to the finite element space. As illustrated using a simple
one-dimensional analysis in Section~\ref{ssec:1D}, and demonstrated with an
extensive set of numerical tests in Section~\ref{sec:Numerical-tests}
for all MsFEM type approches we implemented, the addition of bubble
functions is definitely benefitial for the overall accuracy of the
approach. 

\medskip

The literature on the types of problems and techniques considered here is of course too vast to be recalled here. A
quite general review is contained in our earlier
work~\cite{companion-article}. We however wish to mention here the
references~\cite{arbogast-review,Efendiev-Hou-book} for the general
background on MsFEM, and the works~\cite{cdz,cm,cs,henning,hornung,lions1980}
specifically addressing problems on perforated domains, either from a
theoretical or a numerical standpoint. 

\medskip

The outline of our article is as follows. As already briefly mentioned,
the rest of this introduction, namely Section~\ref{ssec:1D}, is devoted
to the study of a simple one-dimensional situation. From
Section~\ref{sec:presentation2D} on, we work in two dimensions
throughout the article, both for the analysis and for the numerical
tests of the final section. We however emphasize that, of course, the
approach can be applied to the three-dimensional context and that, most
likely, the theoretical analysis we provide here can also be extended to
the three dimensional case (Note that our analysis
in~\cite{companion-article} was performed in both the two and three
dimensional settings). We will not proceed in this direction
here. In addition and for simplicity, we assume that $\Omega$ is a polygonal domain. Section~\ref{sec:presentation2D} presents our finite element
approach and the main result of numerical analysis
(Theorem~\ref{theo:main}) we are able to prove, under restrictive
assumptions made precise below (in particular, periodicity of the
perforations is assumed, although, in practice, the approach is {\em not} restricted to this setting). Section~\ref{sec:preliminaries} prepares
the ground for the proof of this main result, performed in
Section~\ref{sec:proof}. Our final Section~\ref{sec:Numerical-tests}
then presents a comprehensive set of numerical experiments. When using our
MsFEM approach on a perforated domain, there are essentially
three ``parameters'': (i) the boundary conditions imposed to define the
MsFEM basis functions, (ii) the addition, or not, of bubble functions and
(iii) the possible intersections of the perforations with the edges (facets)
of mesh elements. Assessing the validity of our approach requires to
compare it with the other existing approaches for all possible
combinations of the above three ``parameters''. This is what we complete in
Section~\ref{sec:Numerical-tests}. Our tests demonstrate that the
combination of Crouzeix-Raviart type finite elements and bubble
functions allows to outperform all the other existing approaches on the
problem considered here in a way that is essentially insensitive to
intersections of the mesh with the perforations. 

\subsection{A one-dimensional situation}
\label{ssec:1D}

In order to illustrate the specificity of multiscale perforated
problems, and to already discover some interesting features, we first
consider an academic one-dimensional setting. Consider the
one-dimensional version of the boundary value problem~\eqref{eq:genP}
for~$\Omega =(0,L)$, $B_\varepsilon$ the set of
segments $B_\varepsilon =\cup_{j=1}^J (a_j,b_j)$ with
$0<a_1<b_1<a_2<b_2<\cdots <L$. We suppose that the gaps between the
perforations are of length at most $\varepsilon$, that is $a_1 \leq
\varepsilon$, $a_2-b_1 \leq \varepsilon$, $a_3-b_2 \leq \varepsilon$,
\dots, $L-b_J \leq \varepsilon$. Other than that, we do not put any
assumption on the geometry of these one-dimensional perforations. Note
that in particular (and in contrast to the analysis we
perform later on in this 
article) we do \emph{not} assume any periodicity of the
perforations. The weak form of our problem then reads:
find $u \in H_0^1(\Omega_\varepsilon)$ such that
\begin{equation}
\forall v \in H_0^1(\Omega_\varepsilon), 
\quad
a(u,v)
=
\int_{\Omega_\varepsilon} fv,
\label{1Dw}
\end{equation}
where, we recall, $\Omega_\eps = \Omega \setminus
B_\eps$ denotes the perforated domain and where
$$
a(u,v)
=
\int_{\Omega_\varepsilon} u' \, v'.
$$
We now divide $\Omega$ into $N$
segments $K_i=[x_{i-1},x_i]$, $i=1,\ldots,N$, by the nodes 
$0=x_0<x_1<\cdots <x_N=L$, define the mesh size $H=\max
|x_i-x_{i-1}|$, and consider the multiscale finite element space adapted
to the perforated domain
$$
V_H
=
\left\{
\begin{array}{c}
u_H \in C^0(\Omega) \text{ such that $u_H=0$ on $B_\varepsilon \cup
  \partial \Omega$ and}
\\
u_H^{\prime \prime}=C_i \text{ in $K_i \cap \Omega_\varepsilon$,
  $i=1,\ldots,N$, for some constants $C_i$}
\end{array}
\right\}.
$$
Note that the domain $K_i \cap
\Omega_\varepsilon$ may be not connected. We
nevertheless assume that $u_H^{\prime \prime}$ is equal to the {\em
  same} constant $C_i$ on all the connected components of 
$K_i \cap \Omega_\varepsilon$.

\begin{remark}
In the one-dimensional setting, the Crouzeix-Raviart type boundary condition that we consider in this work simply amounts to a continuity condition at the mesh nodes. This is why we require that $u_H \in C^0(\Omega)$ in the above definition of $V_H$. This observation holds for many variants of MsFEM, including the oversampling variant, which, alike the Crouzeix-Raviart variant we introduce here, uses non-conforming finite elements. In this respect, the one-dimensional setting is not typical. 
\end{remark}

The Galerkin approximation of the solution to
problem~\eqref{1Dw} is then introduced as the solution $u_H \in V_H$ to
\begin{equation}
\forall v_H \in V_H,
\quad
a(u_H,v_H) = \int_{\Omega_\varepsilon} f v_H.
\label{1Dh}
\end{equation}
Readers familiar with the MsFEM approach will notice that $V_H$ contains more functions than the usual MsFEM basis set,
which would consist here in taking $u_H^{\prime \prime}= 0$ (rather than an
arbitrary constant $C_i$) on $K_i \cap \Omega_\varepsilon$. 

\medskip

A convenient
generating family for the space $V_H$ may be constructed as follows. First we
associate a function $\Phi_i$ to any internal node $x_i$ by solving
\begin{eqnarray*}
\text{supp}\,\Phi_i & \subset & (x_{i-1},x_{i+1}), 
\\
\Phi_i^{\prime \prime} & = & 0 
\text{ in $(x_{i-1},x_i) \cap \Omega_\varepsilon$ and in
$(x_i,x_{i+1}) \cap \Omega_\varepsilon$}, 
\\
\Phi_i & = & 0 \text{ in $B_\varepsilon$}, 
\\
\Phi_i(x_i) &=& 1\text{ if $x_i \in \Omega_\varepsilon$
or 0 otherwise}.
\end{eqnarray*}
Note that this construction yields $\Phi_i \equiv 0$ if the node $x_i$ lies
inside a perforation (see Fig.~\ref{fig:phi_i}). Second, we associate a
function $\Psi_i$ to any segment $K_i=[x_{i-1},x_i]$ by solving (see
Fig.~\ref{fig:psi_i}) 
\begin{eqnarray*}
\text{supp} \, \Psi_i & \subset & (x_{i-1},x_i), 
\\
-\Psi_i^{\prime \prime} & = & 1 \text{ in $K_i \cap
  \Omega_\varepsilon$}, 
\\
\Psi_i & = & 0 \text{ in $B_\varepsilon$}.
\end{eqnarray*}
The functions $\Phi_i$ ($1 \leq i \leq N-1$) and $\Psi_j$ ($1 \leq j \leq N$) are linearly independent (except for the trivial case when $\Phi_i \equiv 0$), and we obviously have
$$
\text{span} \left\{ \Phi_1,\ldots,\Phi_{N-1},\Psi_1,\ldots,\Psi_N
\right\} \subset V_H.
$$
In turn, any $u \in V_H$ can be written $u = \sum_{j=1}^N C_j \Psi_j + \sum_{i=1}^{N-1} u(x_i) \Phi_i$. We thus have 
$$
V_H = \text{span} \left\{ \Phi_1,\ldots,\Phi_{N-1},\Psi_1,\ldots,\Psi_N
\right\},
$$
which implies that the space $V_H$ is of dimension at most $2N-1$. 

\begin{figure}[htbp]
\psfrag{xim1}{$x_{i-1}$}
\psfrag{xi}{$x_i$}
\psfrag{xip1}{$x_{i+1}$}
\psfrag{phii}{$\Phi_i$}
\centerline{
\includegraphics[width=7truecm]{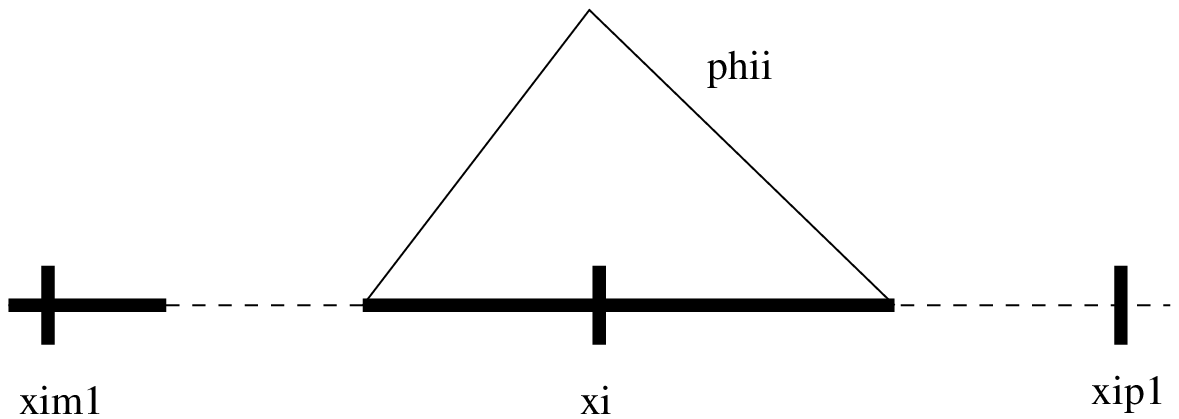}
\hspace{1cm}
\includegraphics[width=7truecm]{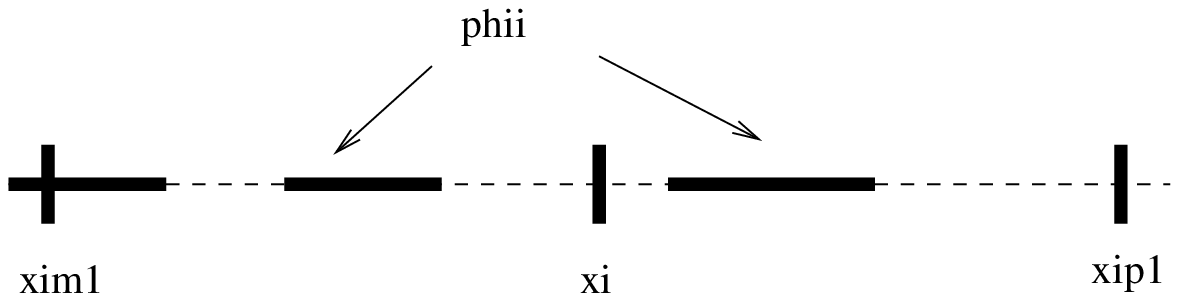}
}
\caption{Basis function $\Phi_i$ (Solid line: domain $\Omega_\eps$;
  dashed line: perforations $B_\eps$). Left: case when 
$x_i \in \Omega_\eps$. Right: case when $x_i \notin \Omega_\eps$, for which
  $\Phi_i \equiv 0$.
\label{fig:phi_i}}
\end{figure}

\begin{figure}[htbp]
\psfrag{xim1}{$x_{i-1}$}
\psfrag{xi}{$x_i$}
\psfrag{xip1}{$x_{i+1}$}
\psfrag{psii}{$\Psi_i$}
\centerline{
\includegraphics[width=7truecm]{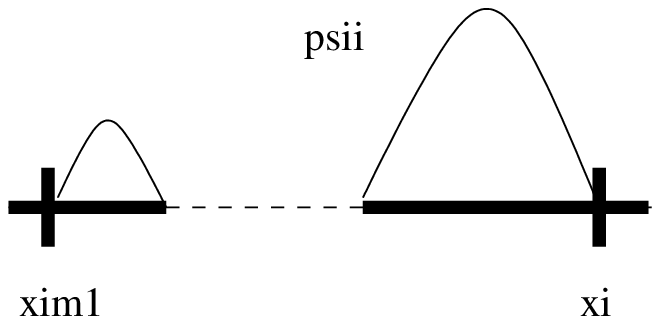}
\hspace{1cm}
\includegraphics[width=7truecm]{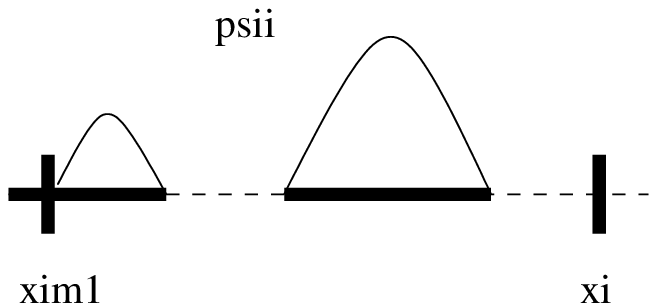}
}
\caption{Basis function $\Psi_i$ (Solid line: domain $\Omega_\eps$;
  dashed line: perforations $B_\eps$). Left: case when 
$x_i \in \Omega_\eps$. Right: case when $x_i \notin \Omega_\eps$. In
both cases, $\Psi_i \neq 0$.
\label{fig:psi_i}}
\end{figure}

It is interesting to note that the functions~$\Psi_i$, which act as
\emph{bubble functions}, are necessary to generate an efficient approximation
space. The reason is evident in the one-dimensional situation,
since, in the absence of such a bubble (or of a basis function playing a
similar role), there is no way to recover a good approximation quality between
two consecutive perforations if no node is actually present there. The
numerical solution would systematically vanish in such a region (see
Fig.~\ref{fig:phi_i}, right part). 
In higher dimensions, the phenomenon is less accute (since perforations,
unless of a particular shape, cannot isolate regions of the space from
the neighborhood) but it is still, to some extent, relevant. We will
observe the definite added value of bubble functions in our numerical
tests of Section~\ref{sec:Numerical-tests}.

\medskip

We then have the following (simple) numerical analysis result.

\begin{theorem}
Assume that the right-hand side $f$ in~\eqref{1Dw} satisfies $f \in
H^1(\Omega)$.  
Then the Galerkin solution $u_H$ of~\eqref{1Dh} satisfies the error estimate 
\begin{equation}
|u-u_H|_{H^1(\Omega_\eps)} \leq C\varepsilon H \| f' \|_{L^2(\Omega)},
\label{eq:1Destimate}
\end{equation}
where $| \cdot |_{H^1(\Omega_\eps)}$ is the energy norm associated to the
bilinear form $a$:
$$
\forall v \in H^1(\Omega_\eps), \quad
| v |_{H^1(\Omega_\eps)} := \sqrt{a(v,v)} = 
\sqrt{\int_{\Omega_\varepsilon} (v')^2}.
$$
\end{theorem}

\medskip

The factor~$\varepsilon$ in the right-hand side of~\eqref{eq:1Destimate}
needs to be understood as follows. It turns out
that, for the category of problems~\eqref{eq:genP} we consider, the
exact solution~$u$ 
(and thus, correspondingly, its numerical approximation~$u_H$) is
of size~$\varepsilon$ in $H^1$ norm for $\varepsilon$ small, as is proved by
homogenization theory and will be recalled --for the periodic setting--
in Section~\ref{ssec:homogenization} below (see~\eqref{eq:lions} and~\eqref{eq:lions-general}). Once 
this scale factor is accounted for, the estimate~\eqref{eq:1Destimate}
shows that the numerical approach is first order accurate in the
meshsize~$H$, with a prefactor $C\, \| f' \|_{L^2(\Omega)}$ that is
\emph{independent} of the size~$\varepsilon$ of the geometric
oscillations.

\medskip

\begin{proof}
We see from~\eqref{1Dw} and~\eqref{1Dh} that
$$
\forall v_H \in V_H, \quad a(u-u_H,v_H) = 0.
$$
Consequently, $u_H$ is the orthogonal projection of $u$ on $V_H$, 
where by \emph{orthogonality} we mean orthogonality for the scalar product
defined by the bilinear form~$a$. We therefore have 
\begin{equation}
|u-u_H|_{H^1(\Omega_\eps)}=\inf_{v_H \in V_H} |u-v_H|_{H^1(\Omega_\eps)}.
\label{cea}
\end{equation}
Proving~\eqref{eq:1Destimate} therefore amounts to proving the
inequality for at 
least one function $v_H \in V_H$. We take $v_H \in V_H$ such that 
$v_H(x_i)=u(x_i)$, $i=0,1,\ldots,N$, and $-v_H^{\prime \prime} = \Pi_H
f$ on each $K_i \cap \Omega_\varepsilon$,
where $\Pi_H f$ is the $L^2$-orthogonal projection of $f$ on the
space of piecewise constant functions. Consider then the
interpolation error $e=u-v_H$. We remark that
$$
\left\{ 
\begin{array}{l}
-e^{\prime \prime}= f-\Pi_H f \ \text{ on each $(x_{j-1},x_j) \cap
\Omega_\varepsilon$, \ \ $1 \leq j \leq N$},
\\ 
e(x_j)=0, \quad j=0,\ldots,N.
\end{array}
\right.
$$
Denoting by $b_0=0$ and $a_{J+1}=L$, we have
\begin{eqnarray}
|e|_{H^1(\Omega_\eps)}^2
=
\int_{\Omega_\varepsilon} |e'|^2
=
\sum_{j=0}^J \int_{b_j}^{a_{j+1}} |e'|^2
&=&
- \sum_{j=0}^J \int_{b_j}^{a_{j+1}} e^{\prime \prime} \, e
\nonumber\\
&=&
\sum_{j=0}^J \int_{b_j}^{a_{j+1}} (f-\Pi_H f) \, e.  
\label{boundseg}
\end{eqnarray}
Note that the integration by parts here does not give rise to any boundary or
jump terms because $e$ vanishes at all the points $a_j$, $b_j$ and also at
the grid points $x_i$ where $e'$ is discontinuous. We now apply the
Cauchy-Schwarz and Poincar\'e inequalities on each segment $(b_j,a_{j+1})$
and note that the constant in the latter inequality scales as the
length of the segment, that is at most $\varepsilon$ (this fact is obvious in dimension $d=1$; see~\eqref{eq:poincare_perfore} and Appendix~\ref{sec:proof_poincare} below for a general argument). 
We thus deduce from~\eqref{boundseg} that
\begin{eqnarray*}
|e|_{H^1(\Omega_\eps)}^2 
& \leq & 
\sum_{j=0}^J \| f - \Pi_H f \|_{L^2(b_j,a_{j+1})} \ 
\| e \|_{L^2(b_j,a_{j+1})} 
\\
& \leq & 
C \varepsilon \sum_{j=0}^J \| f-\Pi_H f \|_{L^2(b_j,a_{j+1})} \ 
|e|_{H^1(b_j,a_{j+1})} 
\\
&\leq &
C \varepsilon \| f-\Pi_H f \|_{L^2(\Omega_\eps)} \ |e|_{H^1(\Omega_\eps)}.
\end{eqnarray*}
Factoring out $|e|_{H^1(\Omega_\eps)}$, and using a standard finite element
approximation estimate of $f-\Pi_H f$, we deduce that
$$
| u - v_H |_{H^1(\Omega_\eps)}
= 
|e|_{H^1(\Omega_\eps)}
\leq 
C \varepsilon \| f - \Pi_H f \|_{L^2(\Omega)}
\leq
C \varepsilon H \| f' \|_{L^2(\Omega)}.
$$
Collecting this bound with~\eqref{cea}, we obtain~\eqref{eq:1Destimate}.
\end{proof}

\section{Presentation of our MsFEM approach in the 2D setting}
\label{sec:presentation2D}

\subsection{MsFEM \`a la Crouzeix-Raviart with bubble functions}
\label{ssec:msfem}

As mentioned in the introduction, we assume henceforth that the ambient
dimension is $d=2$ 
and that $\Omega$ is a polygonal domain. 
We define a mesh $\mathcal{T}_H$ on $\Omega$, i.e. a decomposition of
$\Omega$ into polygons each of diameter at most $H$, and
denote $\mathcal{E}_H$ the set of all the internal edges of
$\mathcal{T}_H$. Note that we mesh $\Omega$ and not the perforated domain
$\Omega_\eps$. This allows us to use coarse elements (independently of the fine scale present in the geometry of $\Omega_\eps$), and leaves us with a lot of flexibility. 
The mesh does not have to be consistent with the perforations $B_\eps$. Some nodes may be in $B_\eps$, and likewise some edges may intersect $B_\eps$.

We also assume that the mesh does not have any hanging
nodes. Otherwise stated, each internal edge is shared by
exactly two elements of 
the mesh. In addition, $\mathcal{T}_H$ is assumed a regular
mesh in the following sense: for any mesh element
$T\in\mathcal{T}_H$, there exists a smooth one-to-one and onto mapping
$K:\overline{T} \to T$ where $\overline{T} \subset \RR^d$ is the
reference element (a polygon of fixed unit
diameter) and $\| \nabla K \|_{L^\infty} \leq CH$, 
$\| \nabla K^{-1} \|_{L^\infty} \leq CH^{-1}$, $C$
being some universal constant independent of~$T$, to which we will refer
as the regularity parameter of the mesh. This assumption is used e.g. in the proof of Lemma~\ref{lem:trace} below.
To avoid some technical complications, we also assume that
the mapping $K$ corresponding to each $T\in\mathcal{T}_H$ is affine on
every edge of $\partial\overline{T}$. Again, this assumption is used e.g. in the proof of Lemma~\ref{lem:trace}. In the following and
to fix the ideas, we will have in mind a
mesh consisting of {\em triangles}, which satisfies the minimum angle
condition to ensure the mesh is regular in the sense defined above (see
e.g.~\cite[Section~4.4]{brenner}). 
We will repeatedly use the notation and
terminology (triangle, \dots) of this setting, although the
approach carries over to quadrangles.

\medskip

The idea behind the MsFEM \`{a} la Crouzeix-Raviart is to 
require the continuity of the (here highly oscillatory) finite element
functions in the sense of averages on the edges. 
We have extensively studied this approach
in~\cite{companion-article}. For the specific setting we 
address here, we add another feature to the numerical approach. Based in
particular on the intuition provided by the one-dimensional case
examined in the previous section, we add \emph{bubble functions} to our
discretization space.

\paragraph{Functional spaces}

To construct our MsFEM space, we proceed as in our previous work~\cite{companion-article}. We introduce the space
$$
W_H = \left\{
\begin{array}{c}
u \in L^2(\Omega) \text{ such that } u|_T \in H^1(T)
\text{ for any } T\in \mathcal{T}_H,
\\ \noalign{\vskip 3pt}
\dps
\int_E [[u]] =0 \text{ for all } E\in \mathcal{E}_H,
\quad u=0 \text{ in $B_\eps$ and on $\partial \Omega$}
\end{array}
\right\},
$$
where~$[[u]]$ denotes the jump of~$u$ across an edge. Note that, as is standard, the condition $u=0$ on $\partial \Omega$ makes sense as
$\Omega$ is a polygonal domain and $\partial \Omega$ belongs to the mesh
edges. We next introduce the subspace
$$
W_H^0 = \left\{
\begin{array}{c}
\dps
u \in W_H \text{ such that } 
\int_E u =0 \text{ for all } E\in \mathcal{E}_H 
\\ \noalign{\vskip 3pt}
\dps
\text{ and } 
\int_T u =0 \text{ for all } T\in \mathcal{T}_H
\end{array}
\right\}
$$
of $W_H$ and define the MsFEM space \`a la Crouzeix-Raviart
\begin{equation}
\label{eq:def_VH}
V_H = \left\{
u \in W_H \text{ such that } 
a_H(u,v) = 0 \text{ for all } v \in W_H^0 
\right\}
\end{equation}
as the orthogonal complement of $W_H^0$ in $W_H$, where by orthogonality we mean orthogonality for the scalar product defined by
\begin{equation}
\label{eq:def_aH}
a_H(u,v) := \sum_{T \in {\cal T}_H} \int_{T \cap \Omega_\varepsilon}
\nabla u \cdot \nabla v.
\end{equation}
We use a broken integral in the definition of $a_H$ since $W_H \not\subset
H^1(\Omega)$.

\paragraph{Notation} 
For any $u \in V_H + H^1_0(\Omega_\eps)$, we denote by
\begin{equation}
\label{eq:def_H1H}
| u |_{H^1_H(\Omega_\eps)} := \sqrt{a_H(u,u)}
\end{equation}
the energy norm associated with the form $a_H$. 

Likewise, for any $u \in H^1_0(\Omega_\eps)$, we denote by
$$
| u |_{H^1(\Omega_\eps)} := \sqrt{\int_{\Omega_\varepsilon}
| \nabla u |^2}
$$
the $H^1$ semi-norm. 

\paragraph{Strong form and basis functions of $V_H$}

Consider any element $T \in \mathcal{T}_H$ (the three edges of which are denoted $\Gamma_i$, $1 \leq i \leq 3$). Taking in the definition of $V_H$ a function $v$ that vanishes on $\Omega \setminus T$, we note that any function $u \in V_H$ satisfies 
$$
\int_{T \cap \Omega_\varepsilon}  \nabla u \cdot \nabla v = 0
$$
for all $v \in H^1(T)$ such that $v=0$ in $B_\eps$, $\dps \int_{\Gamma_i} v = 0$ for all $i$ (if $\Gamma_i \subset \partial \Omega$, the condition 
$\dps \int_{\Gamma_i} v = 0$ is replaced by $v = 0$ on $\Gamma_i$) and 
$\dps \int_T v = 0$. This can be rewritten as
$$
\int_{T \cap \Omega_\varepsilon}  \nabla u \cdot \nabla v 
=
\lambda^T_0 \int_T v +
\sum_{i=1}^3 \lambda^T_i \int_{\Gamma_i} v 
\quad
\text{for all $v \in H^1(T)$ s.t. $v=0$ in $B_\eps$}
$$
for some scalar constants $\lambda^T_j$, $0 \leq j \leq 3$ (on purpose, we have made the dependence of these constants explicit with respect to the mesh element $T$). Hence, the restriction of any $u \in V_H$ to $T$ is in particular a solution to the boundary value problem
\begin{equation}
\label{eq:strong}
-\Delta u = \lambda^T_0 \ \text{in $T \setminus B_\eps$},
\quad 
u = 0 \ \text{in $T \cap B_\eps$},
\quad
n \cdot \nabla u = \lambda^T_i \text{ on each $\Gamma_i$}.
\end{equation}
The flux along each edge interior to $\Omega$ is therefore a constant, the constant being possibly different on the two sides of the edge.

\bigskip

The above observation shows that $V_H$ is a finite dimensional space. We now construct a basis for $V_H$, which consists of functions associated to a particular mesh element or a particular internal edge. Note that no basis function is associated to edges belonging to $\partial \Omega$.

First, for any mesh element $T$ that is not a subset of the perforations $B_\eps$ (i.e. $T \not\subset B_\eps$), we consider the variational problem
\begin{equation}
\label{eq:def_psi}
\inf \left\{ 
\begin{array}{c}
\dps 
\int_{T \setminus B_\eps} 
\left[ \frac12 \left| \nabla \Psi \right|^2 - \Psi \right], 
\ \Psi \in H^1(T), 
\\ \noalign{\vskip 3pt}
\dps 
\Psi = 0 \ \text{in $T \cap B_\eps$}, \
\int_{\Gamma_i} \Psi = 0 \text{ for each $\Gamma_i$}
\end{array}
\right\}.
\end{equation}
Using the Poincar\'e inequality recalled in~\cite[Lemma 9]{companion-article} and standard analysis arguments, we see that this problem has a unique minimizer. We then introduce the function $\Psi_T \in L^2(\Omega)$ which vanishes in $\Omega \setminus T$ and is equal to this minimizer in $T$. We easily deduce from the optimality condition that $\Psi_T \in V_H$ and satisfies
$$
-\Delta \Psi_T = 1 \ \text{in $T \setminus B_\eps$},
\quad 
\Psi_T = 0 \ \text{in $T \cap B_\eps$},
$$
with, for each edge $\Gamma_i$ of $T$, $\dps \int_{\Gamma_i} \Psi_T = 0$ and $n \cdot \nabla \Psi_T = \lambda_i$ on $\Gamma_i$ for some constant $\lambda_i$. 

\medskip

Second, for any internal edge $E$ that is not a subset of the perforations $B_\eps$, we denote $T_E^1$ and $T_E^2$ the two triangles sharing this edge, set $T_E := T_E^1 \cup T_E^2$, and consider the variational problem
\begin{equation}
\label{eq:def_phi}
\inf \left\{ 
\begin{array}{c}
\dps 
\int_{T^1_E \setminus B_\eps} \left| \nabla \Phi \right|^2
+
\int_{T^2_E \setminus B_\eps} \left| \nabla \Phi \right|^2, \ 
\Phi|_{T_E^1} \in H^1(T_E^1), \
\Phi|_{T_E^2} \in H^1(T_E^2),
\\ \noalign{\vskip 3pt}
\dps 
\Phi = 0 \ \text{in $T_E \cap B_\eps$}, \
\int_E \Phi = 1, \ \int_{E'} \Phi = 0 \ \text{for any edge $E' \subset \partial T_E$} 
\end{array}
\right\}.
\end{equation}
This set is not empty due to the fact that $E \not\subset B_\eps$. 
Again, this problem has a unique minimizer. We introduce the function $\Phi_E \in L^2(\Omega)$ which vanishes in $\Omega \setminus T_E$ and is equal to this minimizer in $T_E$. We easily deduce from the optimality condition that $\Phi_E \in V_H$ and satisfies
$$
-\Delta \Phi_E = 0 \ \text{in $T_E^1 \setminus B_\eps$},
\quad 
-\Delta \Phi_E = 0 \ \text{in $T_E^2 \setminus B_\eps$},
\quad 
\Phi_E = 0 \ \text{in $T \cap B_\eps$},
$$
with, for each edge $E' \subset \partial T_E$, $\dps \int_{E'} \Phi_E = 0$ and $n \cdot \nabla \Phi_E = \lambda_{E'}$ on $E'$ for some constant $\lambda_{E'}$ and $\dps \int_E \Phi_E = 1$ and $n \cdot \nabla \Phi_E = \lambda_E$ on $E$ for some constant $\lambda_E$ (with an a priori different constant on the two sides of $E$).

\medskip

For any mesh element $T \subset B_\eps$ (resp. any internal edge $E \subset B_\eps$), we set $\Psi_T \equiv 0$ (resp. $\Phi_E \equiv 0$).

\begin{remark}
In the one-dimensional case, the functions $\Psi_T$ and $\Phi_E$ that we have defined are equal to the basis functions of Section~\ref{ssec:1D} (see Figures~\ref{fig:phi_i} and~\ref{fig:psi_i}).
\end{remark}

The functions $\Psi_T$ and $\Phi_E$ that we have constructed belong to $V_H$. In addition, $\dps \left\{ \Psi_T \right\}_{T \in \mathcal{T}_H, \ T \not\subset B_\eps} \cup \left\{ \Phi_E \right\}_{E \in \mathcal{E}_H, \ E \not\subset B_\eps}$ forms a linearly independent family. We have
$$
\text{Span} 
\left\{ \Phi_E,  \, \Psi_T, \, E\in \mathcal{E}_H, 
\, T\in \mathcal{T}_H \right\}
\subset V_H.
$$
Conversely, let $u \in V_H$. We know that $u$ satisfies~\eqref{eq:strong}. We introduce
$$
v = u - \sum_{T \in \mathcal{T}_H} \lambda_0^T \Psi_T - \sum_{E \in \mathcal{E}_H} \left[ \int_E u \right] \Phi_E
$$ 
and note that it satisfies, for any $T \in \mathcal{T}_H$,
$$
-\Delta v = 0 \ \text{in $T \setminus B_\eps$},
\quad 
v = 0 \ \text{in $T \cap B_\eps$},
$$
with $\dps \int_E v = 0$ and $n \cdot \nabla v$ is a constant on $E$, 
for each edge $E \in \mathcal{E}_H$. This implies that $v \equiv 0$, and thus
\begin{equation}
\label{eq:vh_span}
V_H = 
\text{Span} 
\left\{ \Phi_E,  \, \Psi_T, \, E\in \mathcal{E}_H, 
\, T\in \mathcal{T}_H \right\}.
\end{equation}

\paragraph{Numerical approximation}

The MsFEM approximate solution of our problem~\eqref{eq:genP} is
defined as the solution~$u_H \in V_H$ to 
\begin{equation}
\forall v_H \in V_H, \quad
a_H(u_H,v_H) = \int_{\Omega_\eps} f v_H,
\label{2DH}
\end{equation}
where $a_H$ is defined by~\eqref{eq:def_aH}.
 
\subsection{Main result: an error estimate in the case of periodic perforations}
\label{ssec:mainresult}

The main theoretical result we obtain in this article addresses the
numerical analysis of the approach presented above, in the particular
case of periodic perforations in dimension~2, 
with a sufficient
regularity (made precise in the statement of the theorem below) of the
right-hand side~$f$ of~\eqref{eq:genP}.

\begin{theorem}
\label{theo:main}
Let $u$ be the solution to~\eqref{eq:genP} for $d=2$, with periodic
perforations and with $f \in H^2(\Omega)$.
We assume that, loosely speaking,  the slopes of the mesh edges are
rational numbers. More precisely, we assume that the equation of any
internal edge $E$ of the mesh writes $\dps x_2 = \frac{p_E}{q_E} x_1 + c_E$ for
some $c_E \in \RR$, some $p_E \in \ZZ$ and $q_E \in \NN^\star$ that are
coprime, with 
\begin{equation}
\label{hyp:mesh}
|q_E| \leq C
\end{equation}
for a constant~$C$ independent of the edge considered in the mesh and of
the mesh size $H$. 

Then the MsFEM approximation $u_H$, solution to~\eqref{2DH}, satisfies 
\begin{equation}
|u-u_H|_{H^1_H(\Omega_\eps)} \leq C \varepsilon \, 
\left( \sqrt{\varepsilon} + H + \sqrt{\frac{\eps}{H}} \right) \,
\| f \|_{H^2(\Omega)},
\label{eq:mainresult}
\end{equation}
for some universal constant~$C$ independent from $H$, $\varepsilon$ and $f$,
but depending on the geometry of the mesh and other parameters of the
problem. 
\end{theorem}

As will be evident from the theoretical ingredients recalled below
(see~\eqref{eq:lions} and comments following this 
estimate), the right-hand side of~\eqref{eq:mainresult} needs to be
understood as follows. The size of the exact solution~$u$ (and thus that
of the corresponding approximation~$u_H$) is~$\varepsilon$ in $H^1$
norm. Taking this scale factor into account, the actual rate of convergence for the numerical
approach we design is therefore given by
$\dps \sqrt{\varepsilon}+ H + \sqrt{\eps/H}$.

\begin{remark} Our
  assumption on the rationality of the slopes in the mesh is necessary,
  in the current state of our understanding, to treat traces of periodic
  functions on the edges of the mesh. In full generality, such traces
  are almost periodic functions. Our proof perhaps carries over to this
  case, however at the price of unnecessary technicalities (we refer e.g. to~\cite{gv1,gv2} for works on boundary layers in homogenization, where such non-periodic situations are dealt with). In the case
  of rational slopes we restrict ourselves to, these traces are
  periodic, and the uniform bound~\eqref{hyp:mesh} we additionally assume
  enables us to uniformly bound
  their periods from above, rending the proof much easier. We
  emphasize that our assumption does not seem to us very
  restrictive in practice.  
\end{remark}

\begin{remark}
It is useful to compare our error estimate~\eqref{eq:mainresult} with estimates for other existing MsFEM-type approaches established for similar problems. First, we are not aware of any other numerical analysis of a MsFEM-type approach for problems set on perforated domains. To the best of our knowledge, this work is the first one proposing and analyzing a MsFEM-type approach specifically adapted to such problems.

Second, as pointed out above, this work is a follow up on our previous work~\cite{companion-article} where we have studied a Crouzeix-Raviart type MsFEM approach on the problem
\begin{equation}
\label{eq:pb_ref}
-\dive \left[ A_{\eps}(x) \nabla u^\eps \right] = f 
\text{ in $\Omega$},
\quad
u^\eps = 0 \text{ on $\partial \Omega$},
\end{equation}
the main difference between that method and the one presented here being the addition of bubble functions in the MsFEM space. For problem~\eqref{eq:pb_ref}, we have compared in~\cite[Remark 3.2]{companion-article} our error estimate with those obtained for other MsFEM-type approaches.  
\end{remark}

\begin{remark}
In the absence of perforations, our problem simply writes
$$
-\Delta u = f \text{ in $\Omega$},
\quad
u = 0\text{ on $\partial \Omega$}.
$$
Assuming a triangular mesh is used, our discretization space $V_H$ then becomes the standard Crouzeix-Raviart space~\cite{Crouzeix-Raviart} (see~\cite[Remark 1.1]{companion-article}), complemented by bubble functions defined by~\eqref{eq:def_psi} with $B_\eps = \emptyset$. In turn, the MsFEM approach with linear boundary conditions (as well as the oversampling variant) then becomes the standard P1 FEM.
\end{remark}

\medskip

The next two sections are devoted to the proof of Theorem~\ref{theo:main}. Numerical results are gathered in Section~\ref{sec:Numerical-tests}. 

\section{Some preliminaries}
\label{sec:preliminaries}

\subsection{Elements of homogenization theory for periodically perforated domains}
\label{ssec:homogenization}

We consider the unit square $Y$ and some smooth perforation $B \subset Y$. We
next scale $B$ and $Y$ by a factor $\varepsilon$
and then periodically repeat this pattern with periods $\varepsilon$ in both
directions. The set of perforations is therefore 
$$
B_\eps = \Omega \cap \left( \underset{k \in \ZZ^2}{\cup} \ \eps B_k \right)
\quad \text{with} \quad 
B_k = k + B
$$ 
and the perforated domain is 
$\Omega_\eps = \Omega \setminus \overline{B_\eps}$. 
We denote by $u^\varepsilon$  the solution
to~\eqref{eq:genP} to emphasize the dependency upon~$\varepsilon$. We
know from the classical work~\cite{lions1980}
that, provided $f$ vanishes on the boundary of $\Omega$ (see below the easy adaptation to a more general case), we have
\begin{equation}
\left| u^\varepsilon - \varepsilon^2 
w \left( \frac{\cdot}{\varepsilon} \right) f \right|_{H^1(\Omega_\eps)} 
\leq C \varepsilon^2 \|f\|_{H^2(\Omega)},
\label{eq:lions}
\end{equation}
where $w$ denotes the corrector, that is the solution to the problem 
\begin{eqnarray}
-\Delta w & = &1 \text{ on $Y \setminus B$}, 
\nonumber
\\
w & = & 0\text{ on $\overline{B}$}, 
\label{eq:corrector-lions}
\\
&&\text{$w$ is $Y$-periodic},
\nonumber
\end{eqnarray}
in the unit cell~$Y$. We refer
to~\cite{blp,Engquist-Souganidis,Jikov1994} for more background on
homogenization theory. Note that~\eqref{eq:lions} is {\em not} restricted to the two-dimensional case. In the sequel, we will use the fact that
\begin{equation}
\label{eq:regul_w}
w \in C^1\left(\overline{Y \setminus B}\right),
\end{equation}
which follows from the fact that $w \in C^{2,\delta}\left(\overline{Y \setminus B}\right)$ for some $\delta>0$ (see e.g.~\cite[Theorem 6.14]{gilbarg-trudinger}). In view of~\cite[Corollary 8.11]{gilbarg-trudinger}, we also have $w \in C^\infty(Y \setminus B)$, but we will not need this henceforth.

\medskip

Clearly, \eqref{eq:lions} shows that, for $\varepsilon$ small, the
dominant behaviour of the solution~$u^\varepsilon$ to~\eqref{eq:genP} is
simple. It is obtained by a simple multiplication of the
right-hand side~$f$ by the corrector function. Otherwise stated, the
particular setting yields an homogenized problem where the differential
operator has disappeared. The corrector
problem~\eqref{eq:corrector-lions} formally agrees with intuition: at
the scale of the geometric heterogeneities, the right-hand side~$f$
of~\eqref{eq:genP} is seen as a constant function (thus the right-hand side
of~\eqref{eq:corrector-lions}) and the approximation of the
solution~$u^\varepsilon$ is obtained by the 
simple multiplication mentioned above. Additionally, the ``size'' of the
solution~$u^\varepsilon$ is proportional to~$\varepsilon^2$ in
$L^2$ norm and~$\varepsilon$ in $H^1$ norm, a fact that will need to be
borne in mind below when performing the analysis and the numerical
experiments.

\medskip
 
It is easy to modify~\eqref{eq:lions} in order to accomodate the more
general situation where the right-hand side~$f \in H^2(\Omega)$ does not necessarily
vanish on the boundary of $\Omega$, provided the domain $\Omega$ is smooth. We then have the weaker estimate
\begin{equation}
\left| u^\varepsilon - 
\varepsilon^2 w \left( \frac{\cdot}{\varepsilon} \right) 
f \right|_{H^1(\Omega_\eps)} \leq
C \varepsilon^{3/2} {\cal N}(f),
\label{eq:lions-general}
\end{equation}
where
\begin{equation}
\label{eq:def_N_f}
{\cal N}(f) = \| f \|_{L^\infty(\Omega)} + 
\| \nabla f \|_{L^2(\Omega)} + 
\| \Delta f \|_{L^2(\Omega)}.
\end{equation}
We recall that, in dimension $d=2$, the injection $H^2(\Omega) \subset
C^0(\overline{\Omega})$ is continuous.
The proof of~\eqref{eq:lions-general} is postponed until Appendix~\ref{sec:proof_hom}. 

\medskip

A key ingredient for that proof, and for other proofs throughout this article, is the following Poincar\'e inequality in the perforated domain
$\Omega_\varepsilon$: there exists a constant $C$
independent of $\varepsilon$ such that
\begin{equation}
\label{eq:poincare_perfore}
\forall \phi \in H_0^1(\Omega_\varepsilon), \quad
\| \phi \|_{L^2(\Omega_\varepsilon)}
\leq 
C \varepsilon \| \nabla \phi \|_{L^2(\Omega_\varepsilon)}
=
C \varepsilon |\phi|_{H^1(\Omega_\eps)}.
\end{equation}
The proof of~\eqref{eq:poincare_perfore} is postponed until Appendix~\ref{sec:proof_poincare}. Following the same arguments, we also see that there exists a constant $C$ independent of $\varepsilon$ such that
\begin{equation}
\label{eq:poincare_perfore2}
\forall \phi \in W_H, \quad
\| \phi \|_{L^2(\Omega_\varepsilon)}
\leq 
C \varepsilon |\phi|_{H^1_H(\Omega_\eps)},
\end{equation}
where, we recall, the notation $| \cdot |_{H^1_H(\Omega_\eps)}$ has been defined in~\eqref{eq:def_H1H}.
The condition $\dps \int_E [[\phi]] =0$ (present in the definition of $W_H$) is actually not needed for~\eqref{eq:poincare_perfore2} to hold, given that $\phi =0$ on $B_\eps$.

\subsection{Classical ingredients of multiscale numerical analysis}
\label{ssec:ingredients}

Before we get to the proof of
Theorem~\ref{theo:main}, we first need to collect here some standard
Trace theorems (which were already used and proved in~\cite{companion-article}) and results on the convergence of oscillating
functions. We refer to the textbooks~\cite{brenner,ern,gilbarg-trudinger}
for more details. 
Remark that only Lemma~\ref{lem:malin} is restricted to the
two-dimensional setting.

\medskip
 
First we recall the definition, borrowed from e.g.~\cite[Definition
B.30]{ern}, of the $H^{1/2}$ space. 

\begin{definition}
For any open domain $\omega \subset \RR^n$,
we define the norm
$$
\| u \|^2_{H^{1/2}(\omega)} := 
\| u \|^2_{L^2(\omega)} + | u |^2_{H^{1/2}(\omega)},
$$
where
$$
| u |^2_{H^{1/2}(\omega)} := \int_\omega \int_\omega 
\frac{|u(x)-u(y)|^2}{|x-y|^{n+1}} \, dx dy,
$$
and define the space
$$
H^{1/2}(\omega) := \left\{ u \in L^2(\omega), \quad 
\| u \|_{H^{1/2}(\omega)} < \infty \right\}.
$$
\end{definition}

\paragraph{Trace inequalities}

We have the following trace results:
\begin{lemma}
\label{lem:trace}
There exists $C$ (depending only on the regularity of the mesh)
such that, for any $T \in {\cal T}_H$ and any edge $E
\subset \partial T$, we have 
\begin{equation}
\label{eq:trace1}
\forall v \in H^1(T), \quad
\| v \|^2_{L^2(E)} \leq C \left( 
H^{-1} \| v \|^2_{L^2(T)} + H \| \nabla v \|^2_{L^2(T)}
\right).
\end{equation}
Under the additional assumption that $\dps \int_E v = 0$, we have
\begin{equation}
\label{eq:trace2_pre}
\| v \|^2_{L^2(E)} \leq C H \| \nabla v \|^2_{L^2(T)}
\end{equation}
and
\begin{equation}
\label{eq:trace3_pre}
\| v \|^2_{H^{1/2}(E)} \leq C (1+H) \| \nabla v \|^2_{L^2(T)}.
\end{equation}
\end{lemma}
These bounds
are classical results (see e.g.~\cite[page
282]{brenner}) and are proved
in~\cite[Section 4.2]{companion-article}. The following result is a direct
consequence of~\eqref{eq:trace2_pre} and~\eqref{eq:trace3_pre}: 
\begin{corollary}
\label{coro:trace}
Consider an edge $E \in {\cal E}_H$, and let $T_E \subset {\cal T}_H$
denote all the triangles sharing this edge. 
There exists $C$ (depending only on the regularity of the mesh) such
that
\begin{equation}
\label{eq:trace2}
\forall v \in W_H, 
\quad
\| \, [[v]] \, \|^2_{L^2(E)} \leq C H \sum_{T \in T_E} 
\| \nabla v \|^2_{L^2(T)}
\end{equation}
and
\begin{equation}
\label{eq:trace3}
\forall v \in W_H, 
\quad
\| \, [[v]] \, \|^2_{H^{1/2}(E)} \leq C (1+H) \sum_{T \in T_E} 
\| \nabla v \|^2_{L^2(T)}.
\end{equation}
\end{corollary}

\paragraph{Averages of oscillatory functions}

We shall also need the following classical result.

\begin{lemma}
\label{lem:malin}
Let $g \in L^\infty(\RR)$ be a $q$-periodic function with zero mean. Let
$f \in W^{1,1}(0,H) \subset C^0(0,H)$ be a 
function defined on the interval $[0,H]$ that vanishes at least at one
point of $[0,H]$. Then, for any $\varepsilon >0$,
$$
\left| \int_0^H g\left(\frac{x}{\varepsilon}\right)
f(x)dx \right| 
\leq 
2\varepsilon q \| g \|_{L^\infty(\RR)} \| f' \|_{L^1(0,H)}.
$$
\end{lemma}

\begin{proof} 
The proof is simple and essentially based upon an integration by parts.  
Let $G$ be a primitive of $g$: 
$$
G(x) = \int_0^x g(t)dt. 
$$
The function $G$ is $q$-periodic (as the average of $g$ over its period
vanishes) and bounded, with
$\| G \|_{L^\infty(\RR)} \leq q \| g \|_{L^\infty(\RR)}$. Supposing that the
function $f$ vanishes at the point $c\in [0,H]$, we write
\begin{eqnarray*}
\int_c^H g\left(\frac{x}{\varepsilon}\right) f(x) \, dx
&=&
\int_c^H G'\left(\frac{x}{\varepsilon}\right) f(x) \, dx
\\
&=&
\eps G\left(\frac{H}{\varepsilon}\right) f(H)
-
\int_c^H \eps G\left(\frac{x}{\varepsilon}\right) f'(x) \, dx
\\
&=&
\eps G\left(\frac{H}{\varepsilon}\right) \int_c^H f'(x) \, dx
-
\int_c^H \eps G\left(\frac{x}{\varepsilon}\right) f'(x) \, dx,
\end{eqnarray*}
hence
$$
\left|
\int_c^H g\left(\frac{x}{\varepsilon}\right) f(x) \, dx
\right|
\leq
2 \eps \| G \|_{L^\infty(\RR)} \| f' \|_{L^1(c,H)}
\leq
2 \eps q \| g \|_{L^\infty(\RR)} \| f' \|_{L^1(c,H)}.
$$
By a similar computation,
$$
\left| 
\int_0^c g\left(\frac{x}{\varepsilon}\right) f(x) \, dx
\right| 
\leq 
2 \varepsilon q \|g \|_{L^\infty(\RR)} \| f' \|_{L^1(0,c)}.
$$
The above two bounds imply the result.
\end{proof}

\section{Proof of our main result}
\label{sec:proof}

To prove Theorem~\ref{theo:main}, it is possible to
follow the same arguments as in our earlier
work~\cite{companion-article}. 
We follow here a different
path, so as to show that other strategies are possible. Note that we use here and in~\cite{companion-article} the same technical ingredients, including those recalled in Section~\ref{ssec:ingredients} and an interpolation argument, see Step 1c below.

\medskip

Let $u$ be the solution to the reference problem~\eqref{eq:genP} with the right-hand side~$f$, and
let $\Pi_H f$ be the $L^2$-orthogonal projection of $f$ on the space of
piecewise constant functions. We recall the following standard finite element 
interpolation result: there exists $C$ independent of $H$ and $f$ such that
\begin{equation}
\label{eq:P1_EF}
\| f - \Pi_H f \|_{L^2(\Omega)}
\leq
C H \| \nabla f \|_{L^2(\Omega)}.
\end{equation}
We introduce
\begin{equation}
\label{eq:def_fct_vh}
v_H(x) = \sum_{T \in {\cal T}_H} \Pi_H f \ \Psi_T(x) + \sum_{E \in {\cal E}_H} \left[ \int_E u \right] \ \Phi_E(x),
\end{equation}
where the functions $\Psi_T$ and $\Phi_E$ have been defined in Section~\ref{ssec:msfem} by~\eqref{eq:def_psi} and~\eqref{eq:def_phi} respectively. We recall that, if $T \subset B_\eps$ (resp. $E \subset B_\eps$), then $\Psi_T \equiv 0$ (resp. $\Phi_E \equiv 0$). We see from~\eqref{eq:vh_span} that $v_H \in V_H$. We next decompose the exact solution $u$ of~\eqref{eq:genP} in the form 
$$
u=v_H + \phi.
$$
By definition of $\Psi_T$ and $\Phi_E$, we have, for all edges $E \in {\cal E}_H$ and all triangles $T \in {\cal T}_H$, that
\begin{equation}
\label{eq:par_construction}
\begin{array}{rcl}
\dps \int_E v_H &=& \dps \int_E u \quad \text{hence} \quad \int_E \phi = 0,
\\ \noalign {\vskip 3pt}
\dps n \cdot \nabla v_H &=& \text{Constant on (each side of) $E$},
\\ \noalign {\vskip 3pt}
-\Delta v_H &=& \Pi_H f \text{ on $T \cap \Omega_\varepsilon$}.
\end{array}
\end{equation}
The estimate~\eqref{eq:mainresult} is proved by
estimating $\phi = u-v_H$ in Step 1 below and next $v_H-u_H$ in Step 2. 

In what follows, we use the shorthand notation 
$\dps g_\varepsilon(x) = g \left( x/\eps \right)$ for all functions $g$. The notation $C$ stands for a constant that is independent from $\eps$, $H$, $f$ and $u$, and that may vary from one line to the next. 

\paragraph{Step 1: Estimation of $u-v_H$:}

Using the approximation of $u$ given by the homogenization result~\eqref{eq:lions-general},
we write
\begin{eqnarray}
| \phi |_{H^1_H(\Omega_\eps)}^2 
&=&
\sum_{T \in \mathcal{T}_H} \int_{\Omega_\varepsilon \cap T} 
|\nabla \phi |^2 
\nonumber
\\
&=&
\sum_{T\in \mathcal{T}_H} \int_{\Omega_\varepsilon \cap T} 
\nabla (u-\varepsilon^2 w_\varepsilon f) \cdot \nabla \phi
+
\sum_{T\in \mathcal{T}_H} \int_{\Omega_\varepsilon \cap T}
\nabla (\varepsilon^2 w_\varepsilon f - v_H) \cdot \nabla \phi
\nonumber
\\
&=&
\sum_{T \in \mathcal{T}_H} \int_{\Omega_\varepsilon \cap T}
\nabla (u-\varepsilon^2 w_\varepsilon f) \cdot \nabla \phi
+
\sum_{T\in \mathcal{T}_H} \int_{\Omega_\varepsilon \cap T} 
(-\Delta (\varepsilon^2 w_\varepsilon f - v_H)) \phi
\nonumber
\\
&& + \varepsilon^2 \sum_{T \in \mathcal{T}_H}
\int_{\partial (T \cap \Omega_\varepsilon)} 
\phi \ n \cdot \nabla (w_\varepsilon f)
- 
\sum_{T \in \mathcal{T}_H}
\int_{\partial (T \cap \Omega_\varepsilon)} 
\phi \ n \cdot \nabla v_H.
\label{eq:louis4}
\end{eqnarray}
We now use the fact that $\phi = u - v_H = 0$ on $\partial \Omega_\eps$. We hence have that
\begin{equation}
\label{eq:bord_nul}
\int_{\partial (T \cap \Omega_\varepsilon)} 
\phi \ n \cdot \nabla (w_\varepsilon f)
=
\int_{(\partial T) \cap \Omega_\varepsilon} 
\phi \ n \cdot \nabla (w_\varepsilon f)
\end{equation}
and likewise for the last term of~\eqref{eq:louis4}. Equalities of the type~\eqref{eq:bord_nul} will often be used in the sequel. We thus write~\eqref{eq:louis4} as
\begin{eqnarray*}
| \phi |_{H^1_H(\Omega_\eps)}^2 
&=&
\sum_{T \in \mathcal{T}_H} \int_{\Omega_\varepsilon \cap T}
\nabla (u-\varepsilon^2 w_\varepsilon f) \cdot \nabla \phi
+
\sum_{T\in \mathcal{T}_H} \int_{\Omega_\varepsilon \cap T} 
(-\Delta (\varepsilon^2 w_\varepsilon f - v_H)) \phi
\nonumber
\\
&& + \varepsilon^2 \sum_{T \in \mathcal{T}_H}
\int_{(\partial T) \cap \Omega_\varepsilon} 
\phi \ n \cdot \nabla (w_\varepsilon f)
- 
\sum_{T \in \mathcal{T}_H}
\int_{(\partial T) \cap \Omega_\varepsilon} 
\phi \ n \cdot \nabla v_H.
\end{eqnarray*}
The fourth term in the above right-hand side vanishes. Indeed, on each
edge~$E$, we know from~\eqref{eq:par_construction} that $n \cdot \nabla v_H$ is constant and $\dps \int_E \phi =\int_{E \cap \Omega_\eps} \phi =0$. The third term can be written
$$
\varepsilon^2 \sum_{T \in \mathcal{T}_H}
\int_{(\partial T) \cap \Omega_\varepsilon} 
\phi \ n \cdot \nabla (w_\varepsilon f)
=
\varepsilon^2 \sum_{E \in \mathcal{E}_H}
\int_{E \cap \Omega_\varepsilon} 
[[ \phi ]] \ n \cdot \nabla (w_\varepsilon f).
$$
Indeed, $w \in C^1\left(\overline{Y \setminus B}\right)$ (see~\eqref{eq:regul_w}) and $f \in H^2(\Omega)$, hence $\nabla (w_\varepsilon f)$ has a well-defined trace on $E \cap \Omega_\eps$. 
We are thus left with
\begin{eqnarray}
| \phi |_{H^1_H(\Omega_\eps)}^2 
&=&
\sum_{T \in \mathcal{T}_H} \int_{\Omega_\varepsilon \cap T}
\nabla (u-\varepsilon^2 w_\varepsilon f) \cdot \nabla \phi
+
\sum_{T\in \mathcal{T}_H} \int_{\Omega_\varepsilon \cap T} 
(-\Delta (\varepsilon^2 w_\varepsilon f - v_H)) \phi
\nonumber
\\
&&\qquad + \varepsilon^2 \sum_{E \in \mathcal{E}_H}
\int_{E \cap \Omega_\varepsilon} 
[[\phi]] \ n \cdot \nabla (w_\varepsilon f).
\label{eq:decompo_error}
\end{eqnarray}
We now successively bound the three terms of the right-hand side
of~\eqref{eq:decompo_error}. Loosely speaking:
\begin{itemize}
\item the first term is small because of the homogenization result~\eqref{eq:lions-general}, that states that $\varepsilon^2 w_\varepsilon f$ is indeed an accurate approximation of $u$.
\item the second term is small because, at the leading order term in $\eps$, the first factor in the integrand is equal to
$-\Delta \left( \varepsilon^2 w_\varepsilon f \right) + \Delta v_H \approx f - \Pi_H f$ which is small due to~\eqref{eq:P1_EF}.
\item estimating the third term is more involved. An essential ingredient is the fact that $w$ is a periodic function. We are thus in position to apply our Lemma~\ref{lem:malin}.
\end{itemize}

\paragraph{Step 1a}
The first term of~\eqref{eq:decompo_error} is easily estimated as follows:
\begin{eqnarray}
\left|
\sum_{T \in \mathcal{T}_H} \int_{\Omega_\varepsilon \cap T}
\nabla (u-\varepsilon^2 w_\varepsilon f) \cdot \nabla \phi
\right|
&\leq&
\sum_{T \in \mathcal{T}_H} 
\| \nabla (u-\varepsilon^2 w_\varepsilon f) \|_{L^2(\Omega_\varepsilon \cap T)}
\
\| \nabla \phi \|_{L^2(\Omega_\varepsilon \cap T)}
\nonumber
\\
&\leq&
| u-\varepsilon^2 w_\varepsilon f |_{H^1(\Omega_\eps)}
\
| \phi |_{H^1_H(\Omega_\eps)}
\nonumber
\\
&\leq&
C \varepsilon^{3/2} \, {\cal N}(f) \,  
| \phi |_{H^1_H(\Omega_\eps)},
\label{eq:bound_1}
\end{eqnarray}
where we have used the discrete Cauchy-Schwarz inequality in the second line and the homogenization result~\eqref{eq:lions-general} in the third line. 

\paragraph{Step 1b}
We next turn to the second term of the right-hand side
of~\eqref{eq:decompo_error}, that we write as follows, using the
corrector equation~\eqref{eq:corrector-lions} and~\eqref{eq:par_construction}:
$$
\sum_{T\in \mathcal{T}_H} \int_{\Omega_\varepsilon \cap T} 
(-\Delta (\varepsilon^2 w_\varepsilon f - v_H)) \phi
=
\int_{\Omega _{\varepsilon }}(f 
- 2 \varepsilon (\nabla w)_\varepsilon \cdot \nabla f
- \varepsilon^2 w_\varepsilon \Delta f
- \Pi_H f ) \phi.
$$
We thus obtain
\begin{eqnarray*}
&&
\left|
\sum_{T\in \mathcal{T}_H} \int_{\Omega_\varepsilon \cap T} 
(-\Delta (\varepsilon^2 w_\varepsilon f - v_H)) \phi
\right|
\\
&\leq &
\Big( \| f - \Pi_H f \|_{L^2(\Omega)} + 
2 \varepsilon \| \nabla w \|_{L^\infty}
\| \nabla f \|_{L^2(\Omega)} + 
\eps^2 \| w \|_{L^\infty} \| \Delta f \|_{L^2(\Omega)}
\Big) \,
\| \phi \|_{L^2(\Omega_\varepsilon)} 
\\
&\leq &
C \varepsilon \, \left(
C H \| \nabla f \|_{L^2(\Omega)} + C \varepsilon {\cal N}(f) 
\right)
| \phi |_{H^1_H(\Omega_\eps)},
\end{eqnarray*}
where ${\cal N}(f)$ is defined by~\eqref{eq:def_N_f} and where, in the
last line, we have used~\eqref{eq:P1_EF},~\eqref{eq:regul_w} and~\eqref{eq:poincare_perfore2}.
We deduce that
\begin{equation}
\label{eq:bound_2}
\left|
\sum_{T\in \mathcal{T}_H} \int_{\Omega_\varepsilon \cap T} 
(-\Delta (\varepsilon^2 w_\varepsilon f - v_H)) \phi
\right|
\leq 
C \varepsilon \, (H+\varepsilon) \, {\cal N}(f) \, 
|\phi |_{H^1_H(\Omega_\eps)}.
\end{equation}

\paragraph{Step 1c}
The final stage of Step 1 is devoted to bounding the third
term of the right-hand side of~\eqref{eq:decompo_error}.

In view of the assumptions on the mesh (rationality of the slopes, in short), we first observe that, for any
edge $E \in {\cal E}_H$, the function 
$\dps x \in E \mapsto n \cdot \nabla w\left( \frac{x}{\eps} \right)$ 
is periodic with period $q_E \eps$, for some $q_E \in \NN^\star$
satisfying $|q_E| \leq C$ for some $C$ independent of the mesh edge and of $H$. We
denote by $\langle n \cdot (\nabla w)_\eps \rangle_E$ the average of that
function over one period, and decompose the third
term of the right-hand side of~\eqref{eq:decompo_error} as follows:
\begin{eqnarray}
&&
\varepsilon^2 \sum_{E \in \mathcal{E}_H}
\int_{E \cap \Omega_\varepsilon} 
[[\phi]] \ n \cdot \nabla (w_\varepsilon f)
\nonumber
\\
&=&
\varepsilon \sum_{E\in \mathcal{E}_H}
\int_{E\cap \Omega_\varepsilon} [[\phi]] \ \Big( n \cdot (\nabla
  w)_\varepsilon  - \langle n \cdot (\nabla w)_\eps \rangle_E \Big) f
\nonumber
\\
& + &
\varepsilon \sum_{E\in \mathcal{E}_H} 
\langle n \cdot (\nabla w)_\eps \rangle_E
\int_{E\cap \Omega_\varepsilon} [[\phi]] \ f
+
\varepsilon^2 \sum_{E \in \mathcal{E}_H}
\int_{E \cap \Omega_\varepsilon} 
[[\phi]] \ w_\varepsilon \ n \cdot \nabla f.
\label{eq:terme3}
\end{eqnarray}
We successively estimate the three terms of the right-hand
side of~\eqref{eq:terme3}. In some formulae below, we will make the
following slight abuse of notation. We will extend
the function $\phi = u - v_H$ by 0 inside the perforations
$B_\varepsilon$, so that we can understand~$\phi$ either as a function in
$H_0^1(\Omega_\varepsilon)$ or in $H_0^1(\Omega)$. 

\bigskip

We consider the first term of the right-hand side of~\eqref{eq:terme3}, which we evaluate essentially using the fact that it contains a periodic oscillatory function of zero mean. We claim that
\begin{multline}
\label{eq:claim3_1}
\left|
\int_{E\cap \Omega_\varepsilon} [[\phi]] \ \Big( n \cdot (\nabla
  w)_\varepsilon  - \langle n \cdot (\nabla w)_\eps \rangle_E \Big) f
\right|
\\
\leq
C \, \sqrt{\eps} \, \| f \|_{H^1(E)} 
\| \, [[ \phi ]] \, \|_{H^{1/2}(E)}
\end{multline}
for a constant $C$ independent of the edge $E$, $\eps$ and $H$. Indeed, we first
note that $u$ and $v_H$ vanish on $\Omega \setminus \Omega_\eps$, so
$\phi=u-v_H$ vanishes on $E \cap (\Omega \setminus \Omega_\eps)$, hence
\begin{multline}
\label{eq:debut}
\int_{E\cap \Omega_\varepsilon} [[\phi]] \ \Big( n \cdot (\nabla
  w)_\varepsilon  - \langle n \cdot (\nabla w)_\eps \rangle_E \Big) f
\\
=
\int_E [[\phi]] \ \Big( n \cdot (\nabla
  w)_\varepsilon  - \langle n \cdot (\nabla w)_\eps \rangle_E \Big) f.
\end{multline}
Second, using the regularity~\eqref{eq:regul_w} of $w$, we obviously have that
\begin{equation}
\label{eq:l2}
\left|
\int_E [[\phi]] \ \Big( n \cdot (\nabla
  w)_\varepsilon  - \langle n \cdot (\nabla w)_\eps \rangle_E \Big) f
\right|
\leq
C \, \|f\|_{L^2(E)} \, \| \, [[\phi]] \, \|_{L^2(E)}.
\end{equation}
Third, suppose momentarily that $[[ \phi ]] \in H^1(E) \subset
C^0(E)$. We infer from 
the fact that $\dps \int_E [[\phi]] = 0$ that $[[ \phi ]]$, and hence
$[[ \phi ]] \, f$, vanishes at least at one point on $E$. In addition,
the function $n \cdot (\nabla
  w)_\varepsilon  - \langle n \cdot (\nabla w)_\eps \rangle_E$ 
is periodic on $E$ (with a period $q_E$ uniformly bounded with respect to $E
\in {\cal E}_H$) and of zero mean. We are then in
position to apply Lemma~\ref{lem:malin}, which yields, using~\eqref{eq:regul_w},
\begin{eqnarray}
\left|
\int_E [[\phi]] \ \Big( n \cdot (\nabla
  w)_\varepsilon  - \langle n \cdot (\nabla w)_\eps \rangle_E \Big) f
\right|
& \leq &
4 \, \eps \, q_E \| \nabla w \|_{C^0} 
\| \nabla_E \left( f [[ \phi ]] \right) \|_{L^1(E)}
\nonumber
\\
& \leq &
C \, \eps \, \| f \|_{H^1(E)} 
\| \, [[ \phi ]] \, \|_{H^1(E)},
\label{eq:h1}
\end{eqnarray}
where, for any function $g$, $\nabla_E g = t_E \cdot \nabla g$ where $t_E$ is a unit vector tangential to the edge $E$.
By interpolation between~\eqref{eq:l2} and~\eqref{eq:h1}, and
using~\eqref{eq:debut}, we obtain~\eqref{eq:claim3_1}, with a constant $C$ (independent of the edge) which is independent from $\eps$ and $H$ by scaling arguments (see~\cite{companion-article} for details).

We then deduce from~\eqref{eq:claim3_1} that the first term of the
right-hand side of~\eqref{eq:terme3} satisfies
\begin{eqnarray*}
&&
\left|
\varepsilon \sum_{E\in \mathcal{E}_H}
\int_{E\cap \Omega_\varepsilon} [[\phi]] \ \Big( n \cdot (\nabla
  w)_\varepsilon  - \langle n \cdot (\nabla w)_\eps \rangle_E \Big) f
\right|
\nonumber
\\
&\leq&
C \, \eps^{3/2} \sum_{E\in \mathcal{E}_H} \| f \|_{H^1(E)} 
\| \, [[ \phi ]] \, \|_{H^{1/2}(E)}
\nonumber
\\
&\leq&
C \, \eps^{3/2} 
\left( \sum_{E\in \mathcal{E}_H} \| f \|^2_{H^1(E)} \right)^{1/2}
\left( \sum_{E\in \mathcal{E}_H} \| \, [[ \phi ]] \, \|^2_{H^{1/2}(E)} \right)^{1/2}
\nonumber
\\
&\leq&
C \, \eps^{3/2} 
\left( \sum_{E\in \mathcal{E}_H; \text{choose one $T \in T_E$}} 
\frac{1}{H} \| f \|^2_{H^1(T)} + H \| \nabla f \|^2_{H^1(T)} \right)^{1/2}
\nonumber
\\
&& \times
\left( 
\sum_{E\in \mathcal{E}_H} \sum_{T \in T_E} \| \nabla \phi \|^2_{L^2(T)} 
\right)^{1/2},
\end{eqnarray*}
where we have used~\eqref{eq:trace1} of Lemma~\ref{lem:trace}
and~\eqref{eq:trace3} of Corollary~\ref{coro:trace} (and, we recall, 
$T_E \subset {\cal T}_H$ denotes all the triangles sharing the edge
$E$). 
We therefore obtain that the first term of the
right-hand side of~\eqref{eq:terme3} satisfies
\begin{eqnarray}
&&
\left|
\varepsilon \sum_{E\in \mathcal{E}_H}
\int_{E\cap \Omega_\varepsilon} [[\phi]] \ \Big( n \cdot (\nabla
  w)_\varepsilon  - \langle n \cdot (\nabla w)_\eps \rangle_E \Big) f
\right|
\nonumber
\\
&\leq&
C \, \eps^{3/2} 
\left( 
\frac{1}{H} \| f \|^2_{H^1(\Omega)} + H \| \nabla f \|^2_{H^1(\Omega)} 
\right)^{1/2}
| \phi |_{H^1_H(\Omega_\eps)}
\nonumber
\\
&\leq&
C \, \eps \left( \sqrt{ \frac{\eps}{H} } \,
\| f \|_{H^1(\Omega)} + \sqrt{\eps H} \, \| \nabla f \|_{H^1(\Omega)} 
\right)
| \phi |_{H^1_H(\Omega_\eps)}.
\label{eq:terme3_1}
\end{eqnarray}

\bigskip

The second term of the right-hand side of~\eqref{eq:terme3} has no oscillatory character. It is therefore estimated
using standard arguments for Crouzeix-Raviart finite elements (using that 
$\dps \int_{E\cap \Omega_\varepsilon} [[\phi]] = 0$), and the
regularity of $w$. Introducing, for each edge $E$, the constant 
$\dps c_E = |E|^{-1} \int_E f$, we bound the second term of the
right-hand side of~\eqref{eq:terme3} as follows: 
\begin{eqnarray}
\nonumber
&&
\left|
\varepsilon \sum_{E\in \mathcal{E}_H} 
\langle n \cdot (\nabla w)_\eps \rangle_E
\int_{E\cap \Omega_\varepsilon} [[\phi]] \ f
\right|
\\
\nonumber
&=&
\left|
\varepsilon \sum_{E\in \mathcal{E}_H} 
\langle n \cdot (\nabla w)_\eps \rangle_E
\int_{E\cap \Omega_\varepsilon} [[\phi]] \ (f - c_E)
\right|
\\
\nonumber
& \leq &
C\varepsilon \sum_{E\in \mathcal{E}_H}
\| \, [[\phi ]] \, \|_{L^2(E)} \,
\| f-c_E \|_{L^2(E)}
\\
\nonumber
& \leq &
C\varepsilon 
\left( \sum_{E\in \mathcal{E}_H} \| \, [[\phi ]] \, \|^2_{L^2(E)} \right)^{1/2}
\left( \sum_{E\in \mathcal{E}_H} \| f-c_E \|^2_{L^2(E)} \right)^{1/2}
\\
\nonumber
&\leq &
C\varepsilon 
\left( 
\sum_{E\in \mathcal{E}_H} H \sum_{T \in T_E} \| \nabla \phi
\|^2_{L^2(T)} 
\right)^{1/2}
\left( \sum_{E\in \mathcal{E}_H; \text{choose one $T \in T_E$}} 
H \| \nabla f \|^2_{L^2(T)} \right)^{1/2}
\\
&\leq &
C \varepsilon H 
| \phi |_{H^1_H(\Omega_\eps)} 
\, \| \nabla f \|_{L^2(\Omega)},
\label{eq:terme3_2}
\end{eqnarray}
where we have used~\eqref{eq:regul_w},~\eqref{eq:trace2} of Corollary~\ref{coro:trace}
and~\eqref{eq:trace2_pre} of Lemma~\ref{lem:trace}. 

\bigskip

We are now left with the third term of the right-hand side
of~\eqref{eq:terme3}. This term has a prefactor $\varepsilon^2$ and all we have to prove is that the term itself is bounded. Using again~\eqref{eq:regul_w},~\eqref{eq:trace1} of
Lemma~\ref{lem:trace} and~\eqref{eq:trace2} of
Corollary~\ref{coro:trace}, we obtain
\begin{eqnarray}
\nonumber
&& \left|
\varepsilon^2 \sum_{E \in \mathcal{E}_H}
\int_{E \cap \Omega_\varepsilon} 
[[\phi]] \ w_\varepsilon \ n \cdot \nabla f
\right|
\\
&\leq&
C \varepsilon^2 \sum_{E\in \mathcal{E}_H}
\| \nabla f \|_{L^2(E)} \, \| \, [[\phi]] \, \|_{L^2(E)}
\nonumber
\\
&\leq &
C \varepsilon^2 
\left( \sum_{E\in \mathcal{E}_H} \| \nabla f \|^2_{L^2(E)} \right)^{1/2} 
\left( \sum_{E\in \mathcal{E}_H} \| \, [[\phi]] \, \|^2_{L^2(E)} \right)^{1/2}
\nonumber
\\
&\leq &
C \varepsilon^2 
\left( \frac{1}{H} \sum_{T\in \mathcal{T}_H} \| \nabla f \|^2_{H^1(T)} \right)^{1/2} 
\left( H \sum_{T\in \mathcal{T}_H} \| \nabla \phi \|^2_{L^2(T)} \right)^{1/2}
\nonumber
\\
&\leq &
C \varepsilon^2 \| \nabla f\|_{H^1(\Omega)} \, | \phi |_{H^1_H(\Omega_\eps)}.
\label{eq:terme3_3}
\end{eqnarray}

Collecting~\eqref{eq:terme3}, \eqref{eq:terme3_1}, \eqref{eq:terme3_2}
and~\eqref{eq:terme3_3}, we obtain that the third term of the right-hand
side of~\eqref{eq:decompo_error} satisfies
\begin{multline}
\left|
\varepsilon^2 \sum_{E \in \mathcal{E}_H}
\int_{E \cap \Omega_\varepsilon} 
[[\phi]] \ n \cdot \nabla (w_\varepsilon f)
\right|
\leq 
C \eps \left( 
\sqrt{ \frac{\eps}{H} } \,
\| f \|_{H^1(\Omega)} + \sqrt{\eps H} \, \| \nabla f \|_{H^1(\Omega)} 
\right.
\\
+ 
H \| \nabla f \|_{L^2(\Omega)} 
+
\varepsilon \| \nabla f\|_{H^1(\Omega)}
\Big) | \phi |_{H^1_H(\Omega_\eps)}.
\label{eq:bound_3}
\end{multline}

\paragraph{Conclusion of Step 1:}
Collecting~\eqref{eq:decompo_error},~\eqref{eq:bound_1},~\eqref{eq:bound_2}
and~\eqref{eq:bound_3}, we deduce that
\begin{eqnarray}
| u-v_H |_{H^1_H(\Omega_\eps)}
&=&
| \phi |_{H^1_H(\Omega_\eps)}
\nonumber
\\
&\leq& 
C \eps \left( \sqrt{\varepsilon} + H + \sqrt{\frac{\varepsilon}{H}} \right) 
\left( \|f\|_{L^\infty(\Omega)} + \|\nabla f\|_{H^1(\Omega)} \right).
\label{eq:fin_step1}
\end{eqnarray}
This concludes the first step of the proof.

\paragraph{Step 2: Estimation of $u_H-v_H$:}

Denoting by $\phi_H = u_H-v_H$, where $u_H$ is the solution to~\eqref{2DH} and $v_H$ is defined by~\eqref{eq:def_fct_vh}, we observe that
\begin{equation}
\label{eq:louis1}
| \phi_H |_{H^1_H(\Omega_\eps)}^2 
=
a_H(u_H-v_H,\phi_H)
=
a_H(u-v_H,\phi_H) + a_H(u_H-u,\phi_H),
\end{equation}
where, we recall, $a_H$ is defined by~\eqref{eq:def_aH}. 
The first term is estimated using~\eqref{eq:fin_step1}. 
The main part of this Step is thus devoted to estimating the second term of~\eqref{eq:louis1}. 

Since $\phi_H \in V_H$, we deduce from the discrete variational formulation~\eqref{2DH} that
\begin{eqnarray}
&&
a_H(u_H-u,\phi_H)
\nonumber
\\
&=&
\int_{\Omega_\eps} f \phi_H - \sum_{T \in {\cal T}_H} 
\int_{T \cap \Omega_\eps} \nabla u \cdot \nabla \phi_H
\nonumber
\\
&=&
\int_{\Omega_\eps} f \phi_H
- 
\sum_{T \in {\cal T}_H} 
\int_{T \cap \Omega_\eps} \nabla (u - \eps^2 w_\eps f) \cdot \nabla \phi_H
- 
\eps^2 \sum_{T \in {\cal T}_H} 
\int_{T \cap \Omega_\eps} \nabla (w_\eps f) \cdot \nabla \phi_H
\nonumber
\\
&=&
\int_{\Omega_\eps} f \phi_H 
- 
\sum_{T \in {\cal T}_H} 
\int_{T \cap \Omega_\eps} \nabla (u - \eps^2 w_\eps f) \cdot \nabla \phi_H
\nonumber
\\
&&
\qquad - 
\eps^2 \sum_{T \in {\cal T}_H} 
\int_{\partial (T \cap \Omega_\eps)} \phi_H \, n \cdot \nabla (w_\eps f)
+
\eps^2 \sum_{T \in {\cal T}_H} 
\int_{T \cap \Omega_\eps} \phi_H \Delta (w_\eps f).
\label{eq:louis2_pre}
\end{eqnarray}
Since $\phi_H = 0$ on $\partial \Omega_\eps$, we can take the integral in the third term of~\eqref{eq:louis2_pre} only on $(\partial T) \cap \Omega_\eps$. Using~\eqref{eq:corrector-lions} for the fourth term, we obtain that
\begin{eqnarray}
&&
a_H(u_H-u,\phi_H)
\nonumber
\\
&=&
- 
\sum_{T \in {\cal T}_H} 
\int_{T \cap \Omega_\eps} \nabla (u - \eps^2 w_\eps f) \cdot \nabla \phi_H
- 
\eps^2 \sum_{T \in {\cal T}_H} 
\int_{(\partial T) \cap \Omega_\eps} \phi_H \, n \cdot \nabla (w_\eps f)
\nonumber
\\
&&
\qquad 
+
\eps \sum_{T \in {\cal T}_H} 
\int_{T \cap \Omega_\eps} \phi_H \Big( 
2 (\nabla w)_\eps \cdot \nabla f
+
\eps w_\eps \Delta f \Big).
\label{eq:louis2}
\end{eqnarray}

We now successively bound the three terms of the right-hand side
of~\eqref{eq:louis2}. The first term is estimated simply using homogenization theory, since it is not specifically related to the discretization. We write, as in~\eqref{eq:bound_1},
\begin{equation}
\label{eq:bound_1new}
\left|
\sum_{T \in \mathcal{T}_H} \int_{\Omega_\varepsilon \cap T}
\nabla (u-\varepsilon^2 w_\varepsilon f) \cdot \nabla \phi_H
\right|
\leq
C \varepsilon^{3/2} \, {\cal N}(f) \,  
| \phi_H |_{H^1_H(\Omega_\eps)}.
\end{equation}
For the second term of the right-hand side
of~\eqref{eq:louis2}, we use the same arguments as for the third term of~\eqref{eq:decompo_error}. We have
$$
\eps^2 \sum_{T \in {\cal T}_H} 
\int_{(\partial T) \cap \Omega_\eps} \phi_H \, n \cdot \nabla (w_\eps f)
=
\eps^2 \sum_{E \in {\cal E}_H} 
\int_{E \cap \Omega_\eps} [[\phi_H]] \, n \cdot \nabla (w_\eps f),
$$
and we note that $\dps \int_E [[ \phi_H ]] = 0$. We therefore can use
the same arguments as in Step 1c, and obtain, similarly
to~\eqref{eq:bound_3},
\begin{multline}
\left|
\varepsilon^2 \sum_{E \in \mathcal{E}_H}
\int_{E \cap \Omega_\varepsilon} 
[[\phi_H]] \ n \cdot \nabla (w_\varepsilon f)
\right|
\\
\leq 
C \eps \left( 
\sqrt{ \frac{\eps}{H} } \,
\| f \|_{H^1(\Omega)} 
+ 
(\varepsilon +H) \| \nabla f\|_{H^1(\Omega)}
\right) | \phi_H |_{H^1_H(\Omega_\eps)}.
\label{eq:bound_3new}
\end{multline}

We next turn to the third term of the right-hand side
of~\eqref{eq:louis2}, which is estimated using the Cauchy-Schwarz inequality, the fact that the second factor is bounded and the first factor satisfies a Poincar\'e inequality. Indeed, using the regularity~\eqref{eq:regul_w} of $w$ and
the Poincar\'e inequality~\eqref{eq:poincare_perfore2} satisfied by $\phi_H \in V_H \subset W_H$, we have
\begin{eqnarray}
&&
\left|
\eps \sum_{T \in {\cal T}_H} 
\int_{T \cap \Omega_\eps} \phi_H \Big( 
2 (\nabla w)_\eps \cdot \nabla f
+
\eps w_\eps \Delta f \Big)
\right|
\nonumber
\\
&\leq&
C \eps \sum_{T \in {\cal T}_H} 
\| \phi_H \|_{L^2(T \cap \Omega_\eps)} \, 
\left( \| \nabla f \|_{L^2(T \cap \Omega_\eps)} + \eps 
\| \Delta f \|_{L^2(T \cap \Omega_\eps)} \right)
\nonumber
\\
&\leq&
C \eps \| \phi_H \|_{L^2(\Omega_\eps)} \, 
\| \nabla f \|_{H^1(\Omega)}
\nonumber
\\
&\leq&
C \eps^2 | \phi_H |_{H^1_H(\Omega_\eps)} \, 
\| \nabla f \|_{H^1(\Omega)}.
\label{eq:bound_T3}
\end{eqnarray}
Collecting~\eqref{eq:louis2}, \eqref{eq:bound_1new},
\eqref{eq:bound_3new} and~\eqref{eq:bound_T3}, we deduce that
\begin{multline}
\label{eq:louis3}
\left|
a_H(u_H-u,\phi_H)
\right|
\\
\leq
C \eps \left( \sqrt{\eps} + H +
\sqrt{ \frac{\eps}{H} }
\right)
\left( \| f \|_{L^\infty(\Omega)} + \| \nabla f \|_{H^1(\Omega)}
\right) \,
| \phi_H |_{H^1_H(\Omega_\eps)}.
\end{multline}
Inserting~\eqref{eq:louis3} into~\eqref{eq:louis1}, we have
\begin{eqnarray*}
&&
| \phi_H |_{H^1_H(\Omega_\eps)}^2 
\\
& \leq &
a_H(u-v_H,\phi_H)
+
C \eps \left( \sqrt{\eps} + H +
\sqrt{ \frac{\eps}{H} }
\right)
\left( \| f \|_{L^\infty(\Omega)} + \| \nabla f \|_{H^1(\Omega)}
\right) \,
| \phi_H |_{H^1_H(\Omega_\eps)}
\\
&\leq &
|u-v_H|_{H^1_H(\Omega_\eps)} \, | \phi_H |_{H^1_H(\Omega_\eps)}
\\
&& \qquad + 
C \eps \left( \sqrt{\eps} + H +
\sqrt{ \frac{\eps}{H} }
\right)
\left( \| f \|_{L^\infty(\Omega)} + \| \nabla f \|_{H^1(\Omega)}
\right) \,
| \phi_H |_{H^1_H(\Omega_\eps)}.
\end{eqnarray*}
Factoring out $| \phi_H |_{H^1_H(\Omega_\eps)}$, and
using~\eqref{eq:fin_step1}, we deduce that
\begin{eqnarray}
| u_H-v_H |_{H^1_H(\Omega_\eps)}
&=&
| \phi_H |_{H^1_H(\Omega_\eps)}
\nonumber
\\
&\leq& 
C \eps \left( \sqrt{\varepsilon} + H + \sqrt{\frac{\varepsilon}{H}} \right) 
\left( \| f \|_{L^\infty(\Omega)} + \| \nabla f \|_{H^1(\Omega)}
\right).
\label{eq:fin_step2}
\end{eqnarray}

\paragraph{Conclusion} We deduce from~\eqref{eq:fin_step1},
\eqref{eq:fin_step2} and the triangle inequality that
$$
| u-u_H |_{H^1_H(\Omega_\eps)}
\leq 
C \eps \left( \sqrt{\varepsilon} + H + \sqrt{\frac{\varepsilon}{H}} \right) 
\left( \| f \|_{L^\infty(\Omega)} + \| \nabla f \|_{H^1(\Omega)}
\right),
$$
which is the desired estimate~\eqref{eq:mainresult}. This concludes the proof of Theorem~\ref{theo:main}.

\section{Numerical tests}
\label{sec:Numerical-tests}

We now solve~\eqref{eq:genP} for some particular settings, comparing our approach with other existing MsFEM type methods. As pointed out in the introduction, we numerically explore the influence of three parameters: 
\begin{itemize}
\item (i) the boundary conditions imposed to define the MsFEM basis functions and (ii) the addition, or not, of bubble functions. To do so, in Section~\ref{sec:num-comp}, we compare the approach we propose with other existing approaches, considering two versions of each approach, one with and the other without bubble functions. 
\item (iii) the possible intersections of the perforations with the edges of mesh elements. We address this question in Section~\ref{sec:num-rob}, and check there the robustness of our approach with respect to the
location of the perforations: the fact
that the mesh intersects, or does not intersect, the perforations has a
very little influence on the (good) accuracy of our approach, in
contrast to other approaches. 
\end{itemize}
We eventually turn in Section~\ref{sec:non-per} to a non-periodic
test-case, where we again show the excellent performance of our approach. 

\medskip

We mention that, in all our numerical experiments, we actually
do not directly solve~\eqref{eq:genP} but a penalized version of this
problem: find $u\in H_0^1(\Omega)$
such that  
$$
-\dive (\nu \nabla u)+\sigma u=f
$$
with the following penalization parameters:
$$
\nu =\left\{ 
\begin{array}{c}
1 \text{ in $\Omega \setminus B_\varepsilon$} 
\\ 
\dps \frac{1}{h} \text{ in $B_\varepsilon$}
\end{array}
\right. 
\quad \text{and} \quad
\sigma =\left\{ 
\begin{array}{c}
0 \text{ in $\Omega \setminus B_\varepsilon$} 
\\ 
\dps \frac{1}{h^3} \text{ in $B_\varepsilon$}
\end{array}
\right.
,
$$
where $h$ is the fine-scale mesh size used to precompute the highly
oscillatory basis functions (see~\cite{angot,carballal} for more
details on the penalization approach and on the above choice of $\nu$ and $\sigma$). In practice, the chosen fine-scale mesh size always satisfies $h \leq \eps/10$.

Note that, because we use a penalized approach, we do not have to mesh $\Omega_\eps$, which could be cumbersome and could possibly request elements of small size (comparable to the small size $\eps$ present in the geometry of $\Omega_\eps$). In addition, if we were working with a mesh of $\Omega_\eps$, we might face difficulties with the oversampling variant of the MsFEM approach that we compare here with our approach. Indeed, edges of the oversampling domain may intersect the perforations. Properly defining the MsFEM basis functions in such a case would not be straightforward. For these two reasons, we consider a penalization approach.

\subsection{Comparison with existing approaches}
\label{sec:num-comp}

We solve~\eqref{eq:genP} on the domain
$\Omega=(0,1)^2$, with the 
right-hand side~$\dps f(x,y) =\sin \frac{\pi x}{2} \, \sin \frac{\pi
  y}{2}$, and we take $B_\varepsilon$ the set of discs of radius
$0.35\varepsilon$ periodically located on the regular grid of period
$\varepsilon =0.03$. 
For the reference solution, we use a mesh of size
  $1024 \times 1024$.


\medskip

The approaches we compare our approach with are the following four respective approaches:
\begin{itemize}
\item the standard Q1 finite element method on the coarse mesh of size $H$. Of course, we do not expect that method to perform well for this multiscale problem and we only consider it as a ``normalization''.
\item the MsFEM with linear boundary conditions. Although this method is now a bit outdated, it is still considered as the primary MsFEM approach, upon which all the other variants are built.
\item the MsFEM with oscillatory boundary conditions.
This variant (in the form presented in~\cite{hou1999}) is restricted to
the two-dimensional setting. It 
uses boundary conditions provided by the solution to
the oscillatory ordinary differential equation obtained by taking the
trace of the original equation on the edge considered. The approach performs fairly well on a number of cases, although it may also fail.
\item the variant of MsFEM using oversampling. This variant is often considered as the ``gold standard'', although it includes a parameter (the oversampling ratio), the value of which should be carefully chosen. When this parameter is taken large, the method becomes (possibly prohibitively) expensive.  
\end{itemize}
In addition, we consider for each of those approaches, and for our
specific Crouzeix-Raviart type approach, two variants:
one with, and the other without a specific enrichment of the basis set
elements using bubble functions. For all approaches but the
Crouzeix-Raviart type approach that we propose, the bubble $\Psi$ on the
quadrangle $Q$ is defined as the solution to
$$
-\Delta \Psi = 1 \text{ on $Q \cap \Omega_\varepsilon$}, 
\quad  \Psi = 0 \text{ on $\partial(Q \cap \Omega_\varepsilon)$}.
$$
For the Crouzeix-Raviart approach, the bubble function $\Psi$ has been defined in Section~\ref{ssec:msfem} by~\eqref{eq:def_psi}.

\begin{remark}
Other variants of the MsFEM approach have also been proposed, such as the 
Petrov-Galerkin variant with oversampling~\cite{hou2004}. We do not consider this variant here, and refer to our previous work~\cite{companion-article} for some elements of comparison (in a slightly different context).
\end{remark}

For a given mesh size $H$, the cost for computing the basis functions (offline stage) varies from one MsFEM variant to the other. However, for a fixed $H$, all methods without (respectively, with) bubble functions essentially share the same cost to solve the macroscopic problem on $\Omega$ (online stage). More precisely, for a given cartesian mesh, and when using variants including the bubble functions, there are 1.5 times more degrees of freedom in our Crouzeix-Raviart approach than in the three alternative MsFEM approaches mentioned above. Since a logarithmic scaling is used for the x-axis in the figures below, this extra cost does not change the qualitative conclusions we draw below.

\medskip

The numerical results we have obtained in the regime where the meshsize $H$
is of the order of, or larger than, the parameter $\eps$ are presented
on~Figure~\ref{fig:errors}. For all values of the meshsize~$H$, and for both $L^2$ and broken $H^1$ norms, a
definite superiority of our approach over 
all other approaches is observed, and the interest of adding bubble
functions to the basis set is, for each approach, also evident.

\medskip

A side remark is the following. On Figure~\ref{fig:errors}, we observe that, when using bubble functions, the error decreases as $H$ increases. This might seem counterintuitive at first sight. Note however that, when $H$ increases, the cost of computing each basis function increases, as we need to solve a local problem (discretized on a mesh of size $h$ controlled by the value of $\eps$) on a larger coarse element. In contrast to traditional FEM, increasing $H$ does not correspond to reducing the overall computational cost. For MsFEM approaches, increasing $H$ actually corresponds to decreasing the online cost but increasing the offline cost. The regime of interest is that of moderate values of $H$, for which the offline stage cost is acceptable. We only show the right part of Figure~\ref{fig:errors} (corresponding to large values of $H$, leading to a prohibitively expensive offline stage) for the sake of completeness. 

\begin{figure}[htbp]
  \centering
\includegraphics[width=6.7truecm]{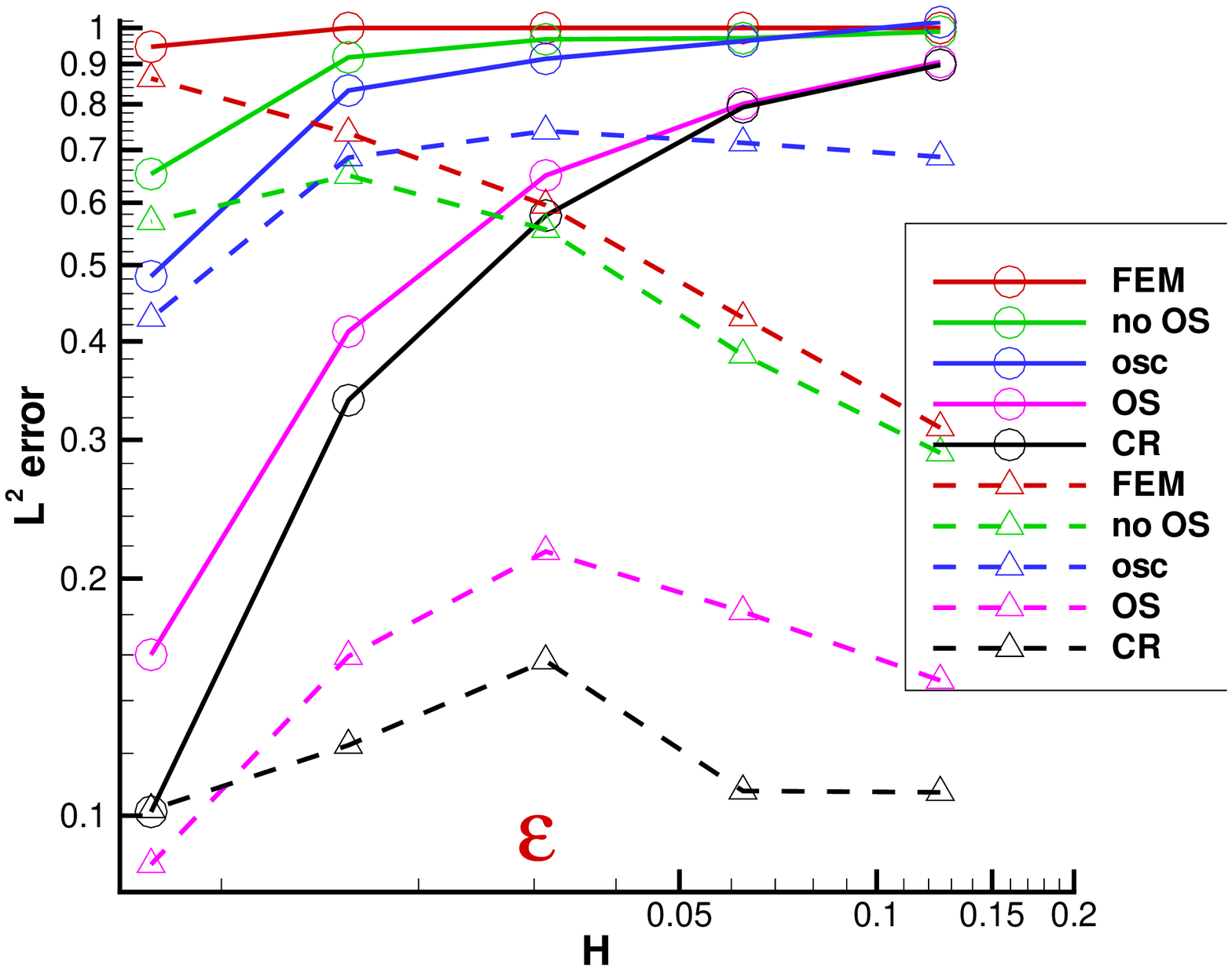}
\includegraphics[width=6.7truecm]{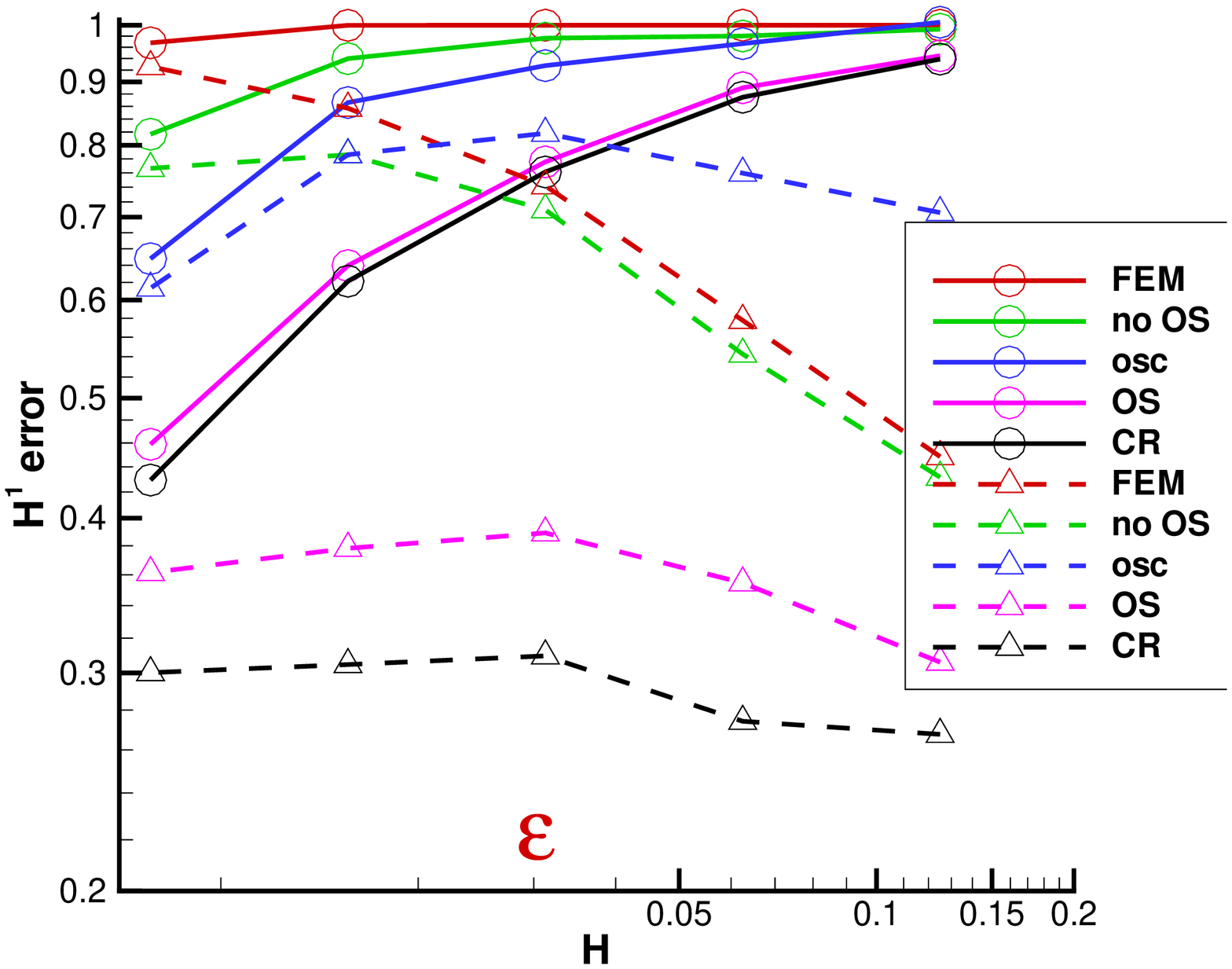}
\caption{Relative ($L^2$, left, and $H^1$-broken, right) errors with
  various approaches in the regimes $H \simeq \eps$ and $H \gtrsim \eps$:
FEM -- the standard Q1 finite elements, no OS -- MsFEM with linear boundary
conditions, osc -- MsFEM with oscillatory boundary
conditions, OS -- MsFEM with oversampling (where the size of the
quadrangles used to compute the basis functions is $3H \times 3H$), CR
-- the MsFEM approach \`a la Crouzeix-Raviart we propose. Results for all 
these methods are represented by solid lines. The dashed lines correspond to
the variants of these methods where we enrich the finite element spaces using bubble functions.
\label{fig:errors}}
\end{figure}

\bigskip

To get a better understanding of the
approaches with bubble functions, we have run a series of tests in a regime
different from that of Figure~\ref{fig:errors}, where the meshsize $H$
is of the order of, or larger than, the parameter $\eps$. On
Figure~\ref{fig:errors_regime}, we present results corresponding to the
regime $H \ll \eps$. This is performed only for the purpose of analyzing the behaviour of the methods and this is of course not the practical regime where we want
to use MsFEM approaches. It is
however useful to observe how the various numerical approaches behave in
that regime. 
We consider the same problem as above, with $\eps= 0.3$ instead of 0.03,
and where the meshsize $H$ ranges from $1/8$ to $1/128$, so that indeed
$H$ is smaller (and even much smaller) than $\eps$. The reference
solution is again computed on a mesh of size 
$1024 \times 1024$. As expected, we then observe that all errors
uniformly decrease when $H$ decreases, in contrast to the situation displayed on Figure~\ref{fig:errors} and commented upon above.
We then recover the classical
behavior of numerical approaches in the limit of fine discretizations.

\begin{figure}[htbp]
  \centering
\includegraphics[width=6.7truecm]{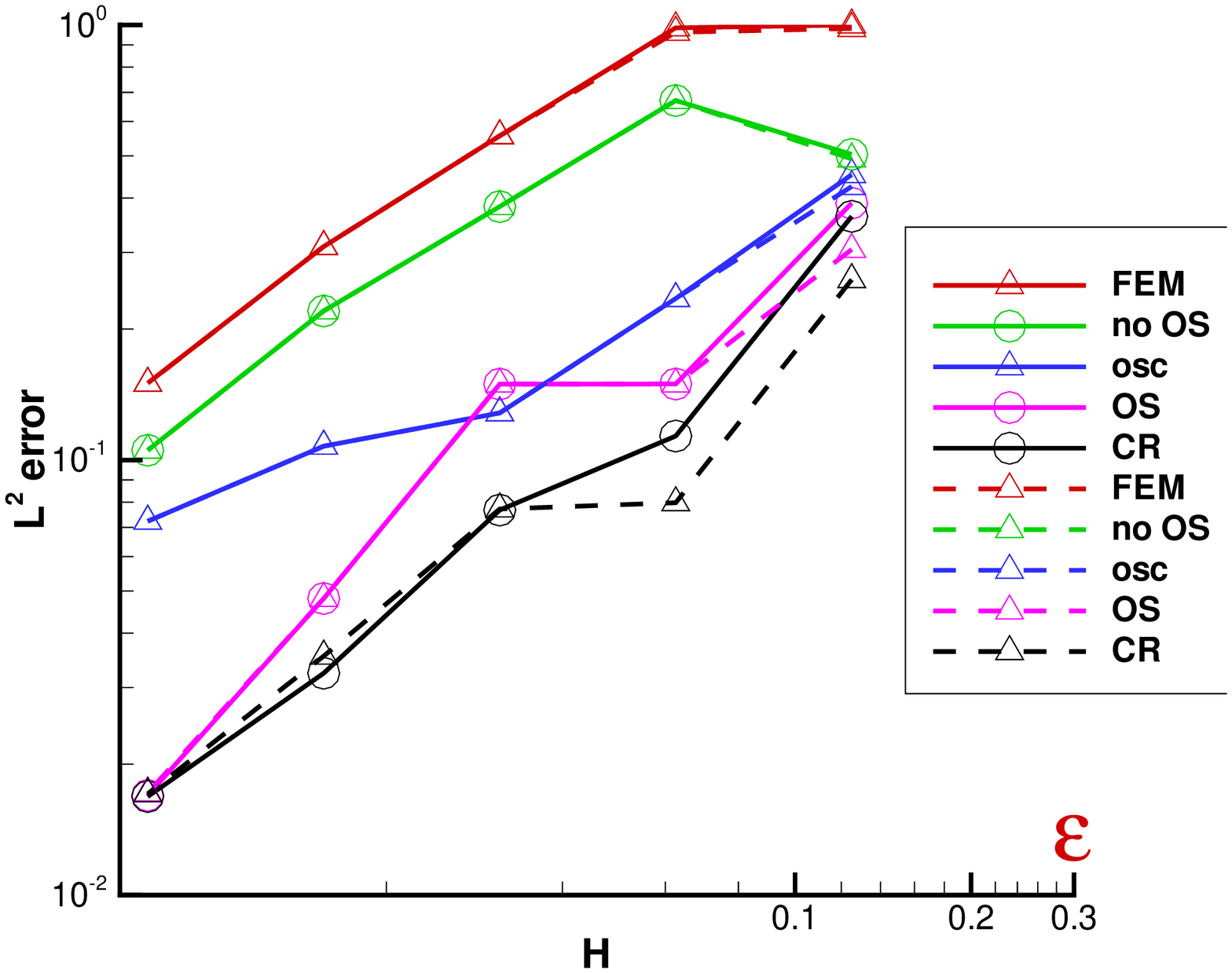}
\includegraphics[width=6.7truecm]{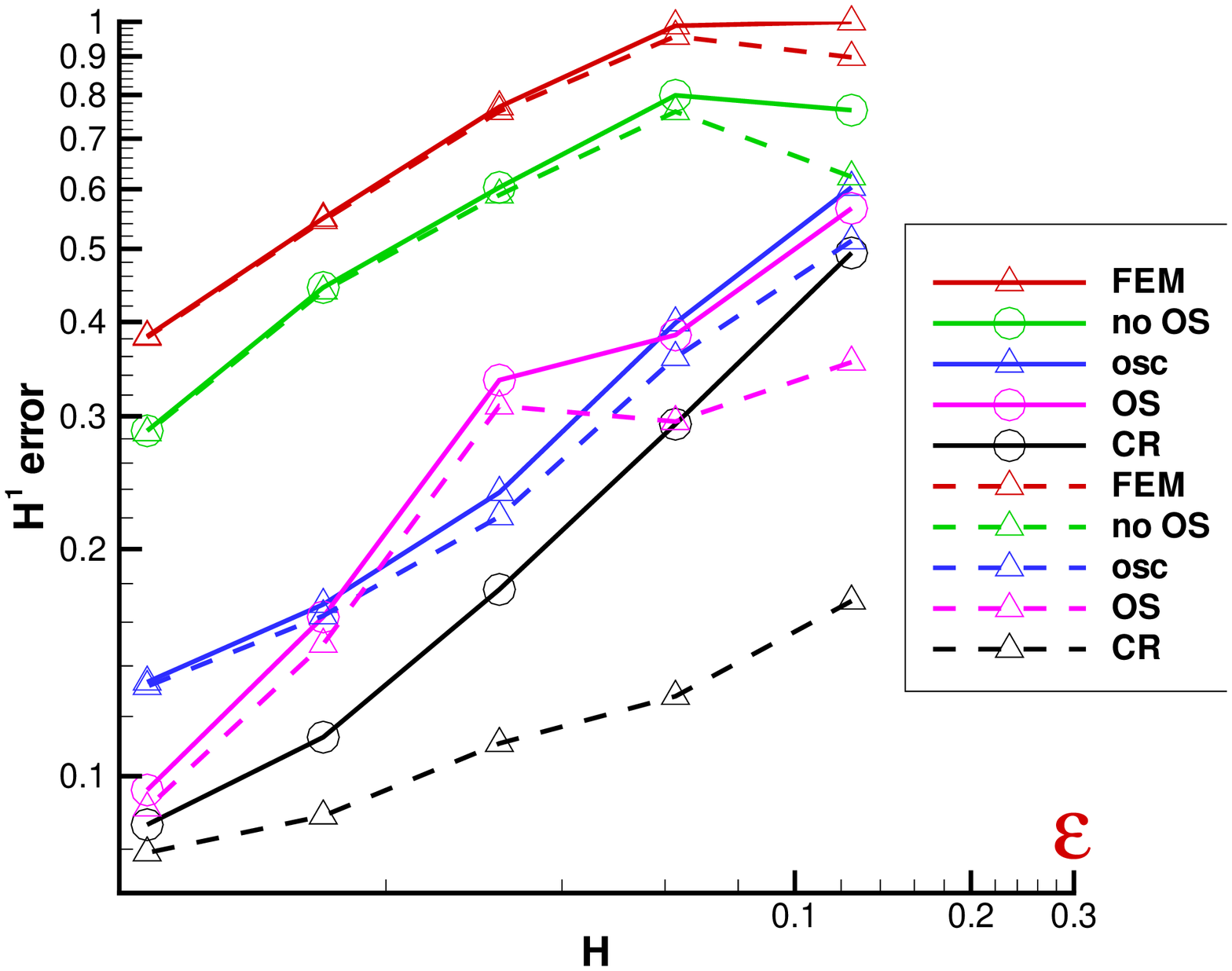}
\caption{Relative ($L^2$, left, and $H^1$-broken, right) errors with
  the same approaches as on Figure~\ref{fig:errors}, in the regime when
  $H \ll \eps$.
\label{fig:errors_regime}}
\end{figure}

\bigskip

For the sake of completeness, we have also considered another oversampling ratio for the MsFEM oversampling approach we compare our approach with. Recall indeed that, on Figures~\ref{fig:errors} and~\ref{fig:errors_regime}, we have considered an oversampling ratio equal to 3. We now additionally consider the method with an oversampling ratio equal to 2. Results are reported on Figure~\ref{fig:over}. As expected, the
accuracy of MsFEM increases when the oversampling ratio increases. The
artificial Dirichlet boundary conditions used to define basis functions
are then further away from the relevant part of the mesh element, and their potentially poor behavior close to the boundary has a smaller influence. Of course, as the 
oversampling ratio increases, the cost of computing these basis
functions increases. We observe that, with the MsFEM approach \`a la
Crouzeix-Raviart we propose, we obtain a better accuracy (again both in $H^1$ and $L^2$ norms) than with
the MsFEM approach that uses an oversampling ratio of 3 (i.e., that
computes basis functions by solving local problems on quadrangles of
size $3H \times 3H$). 

\begin{figure}[htbp]
  \centering
\includegraphics[width=6.7truecm]{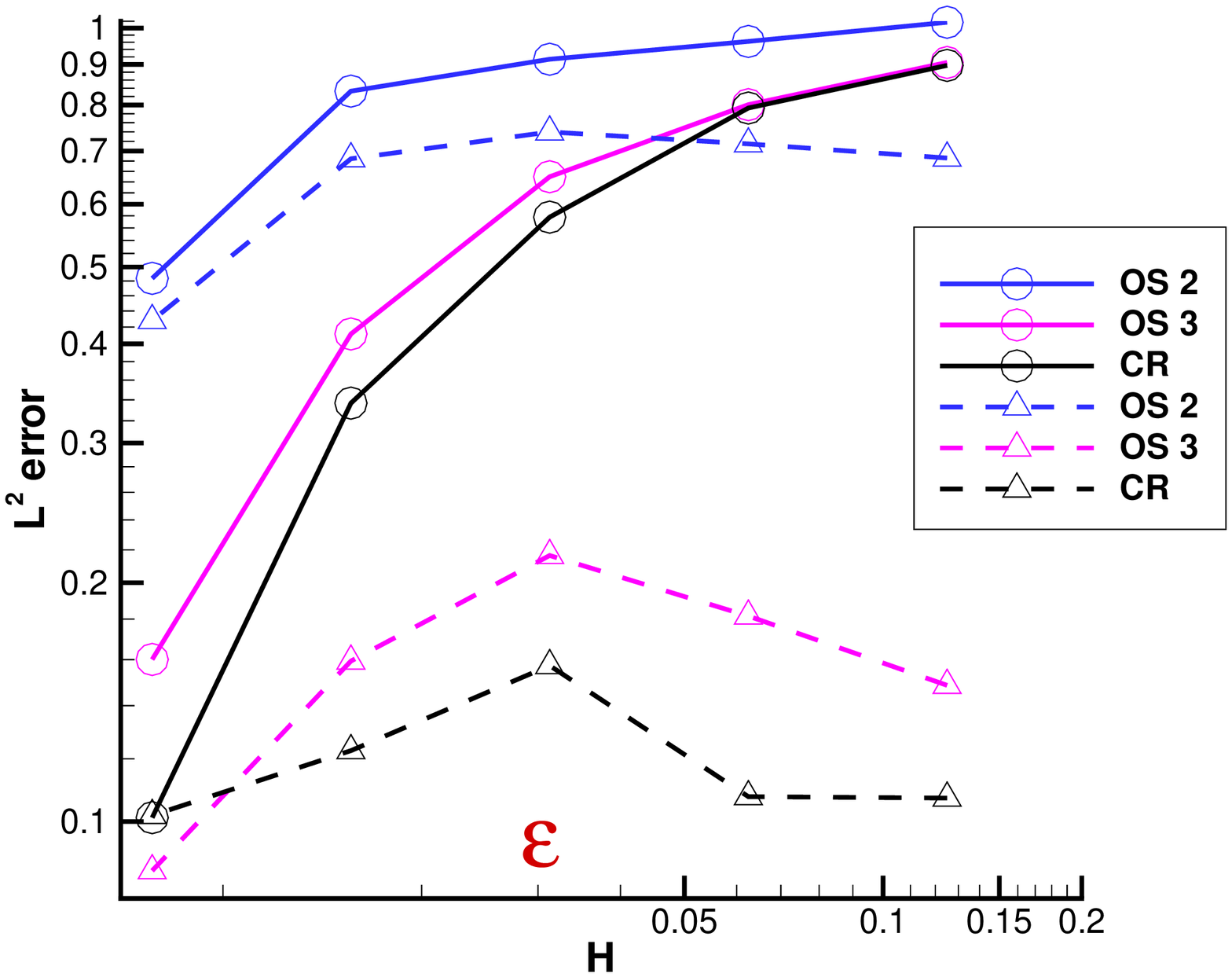}
\includegraphics[width=6.7truecm]{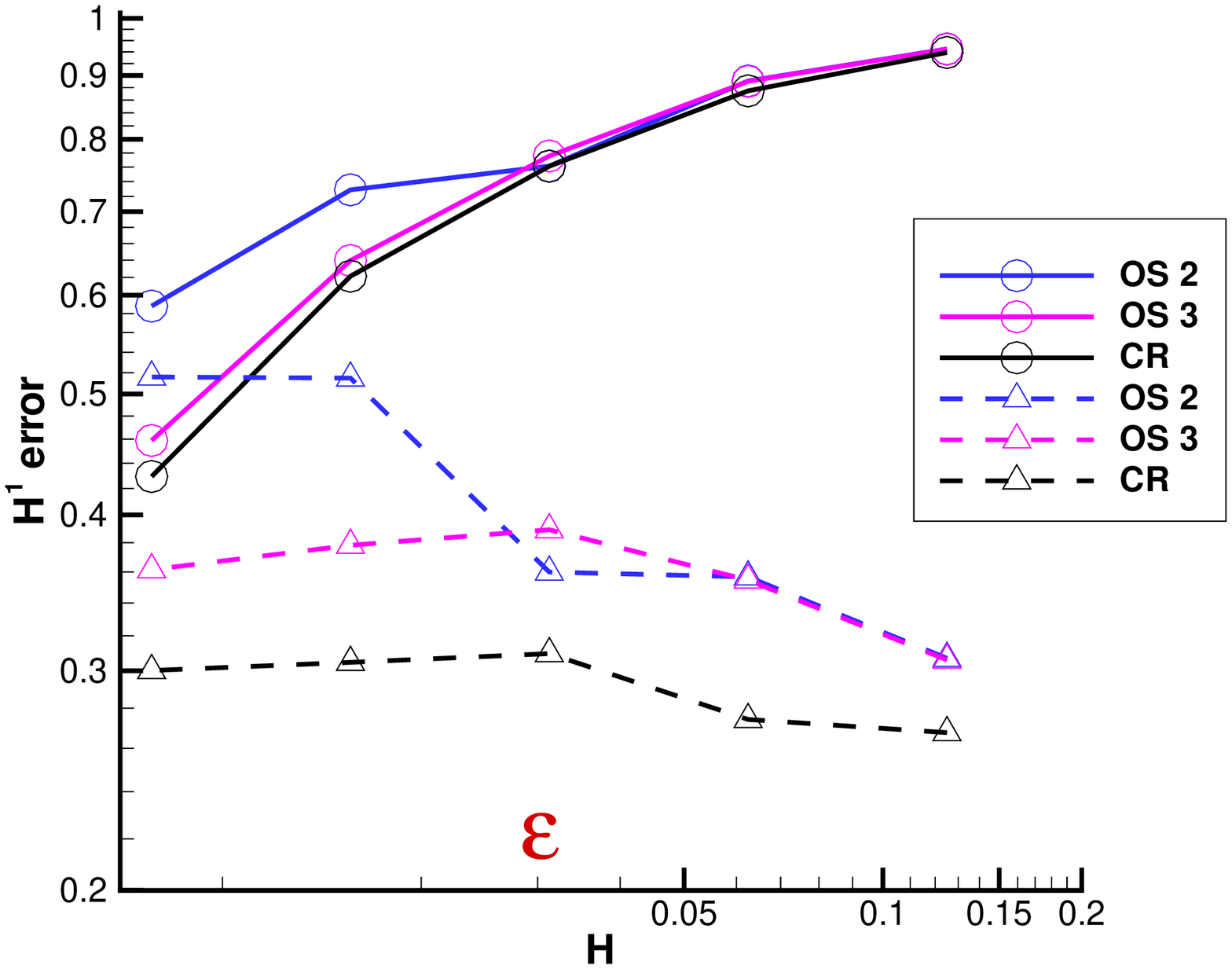}
\caption{Relative ($L^2$, left, and $H^1$-broken, right) errors with
  various approaches (dashed lines: using bubble functions; solid lines:
  without bubble functions): 
OS -- MsFEM with various oversampling ratios, CR -- the MsFEM approach \`a la
Crouzeix-Raviart we propose. 
\label{fig:over}}
\end{figure}

\begin{remark}
Figures~\ref{fig:errors},~\ref{fig:errors_regime} and~\ref{fig:over} show that, for any of the numerical approaches we have considered, the relative $L^2$ error is always smaller than the relative $H^1$ error. The former presumably converges with a better rate (in terms of $\eps$ and $H$) than the latter, although establishing sharp $L^2$ error estimates for MsFEM-type approaches is quite involved (see e.g.~\cite{hou1999}).
\end{remark}

\subsection{Robustness with respect to the location of the perforations}
\label{sec:num-rob}

In this section and in the following one, we perform a series of tests with a different, specific purpose. As
a major motivation for advocating our approach is the flexibility of
Crouzeix-Raviart type finite elements in terms of boundary conditions,
we expect our approach to be particularly effective (and therefore
considerably superior to other approaches) when some edges of
the mesh happen to intersect perforations of the domain. The more such
intersections, the more important the difference. In order to check this expected behaviour, we design the
following test. 

We solve~\eqref{eq:genP} on the domain $\Omega=(0,1)^2$, with a constant
right-hand side~$f=1$, and we take $B_\varepsilon$ the
set of discs of radius $0.2\varepsilon$ periodically located on the
regular grid of period $\varepsilon =0.1$. We compute the reference
solution, and consider 3 variants of MsFEM: the linear
version, the oversampling version and the Crouzeix-Raviart version.
The last three approaches are implemented in the variant that includes
bubble functions in the basis set and they are run on a mesh of size $H=0.2$.

We now perform two sets of numerical experiments. They are identical
except for what concerns the relative position of the mesh with the
perforations. The difference between the two sets of tests is that, from
one set of tests to the other one, the perforations are shifted by
$\varepsilon/2$ in the directions $x$ and $y$. In our Test~1, no edge
intersects any perforation, while, on our Test~2, many edges actually
intersect perforations. 
To some extent, the situation of Test 1 is the best case scenario (where as few edges as possible intersect the perforations) and the other situation is the worst case scenario. 

The numerical  solutions computed for each of the situations considered
is shown on Figures~\ref{fig:decalage-test1}
and~\ref{fig:decalage-test2}, for Test~1 and Test~2 respectively.
The numerical errors observed, computed both in $L^2$ and $H^1$-broken
norms, are correspondingly displayed on Tables~\ref{table:test1}
and~\ref{table:test2} respectively. More than the actual values obtained
for each case, this is the trend of difference between Table~\ref{table:test1}
and Table~\ref{table:test2} that is the practically relevant feature. A
comparison between the two tables indeed show that, qualitatively and in
either of the norms used for measuring the error, the
linear version and the oversampling version of MsFEM are both much more
sensitive to edges intersecting perforations than the Crouzeix-Raviart
version of MsFEM. 
In particular, the gain of our approach with respect to the linear version of MsFEM is much higher in our Test 2 (which is, from the geometrical viewpoint, the worst case scenario) than in Test 1. 
This confirms the intuition of a better flexibility of
our approach. This also allows for expecting a much better
behaviour of that approach for nonperiodic multiscale perforated
problems for which it is extremely difficult, practically, to avoid
repeated intersections of perforations with mesh edges. This is
confirmed by our numerical experiments of Section~\ref{sec:non-per}.

\begin{figure}[htbp]
  \centering
\includegraphics[width=6.7truecm]{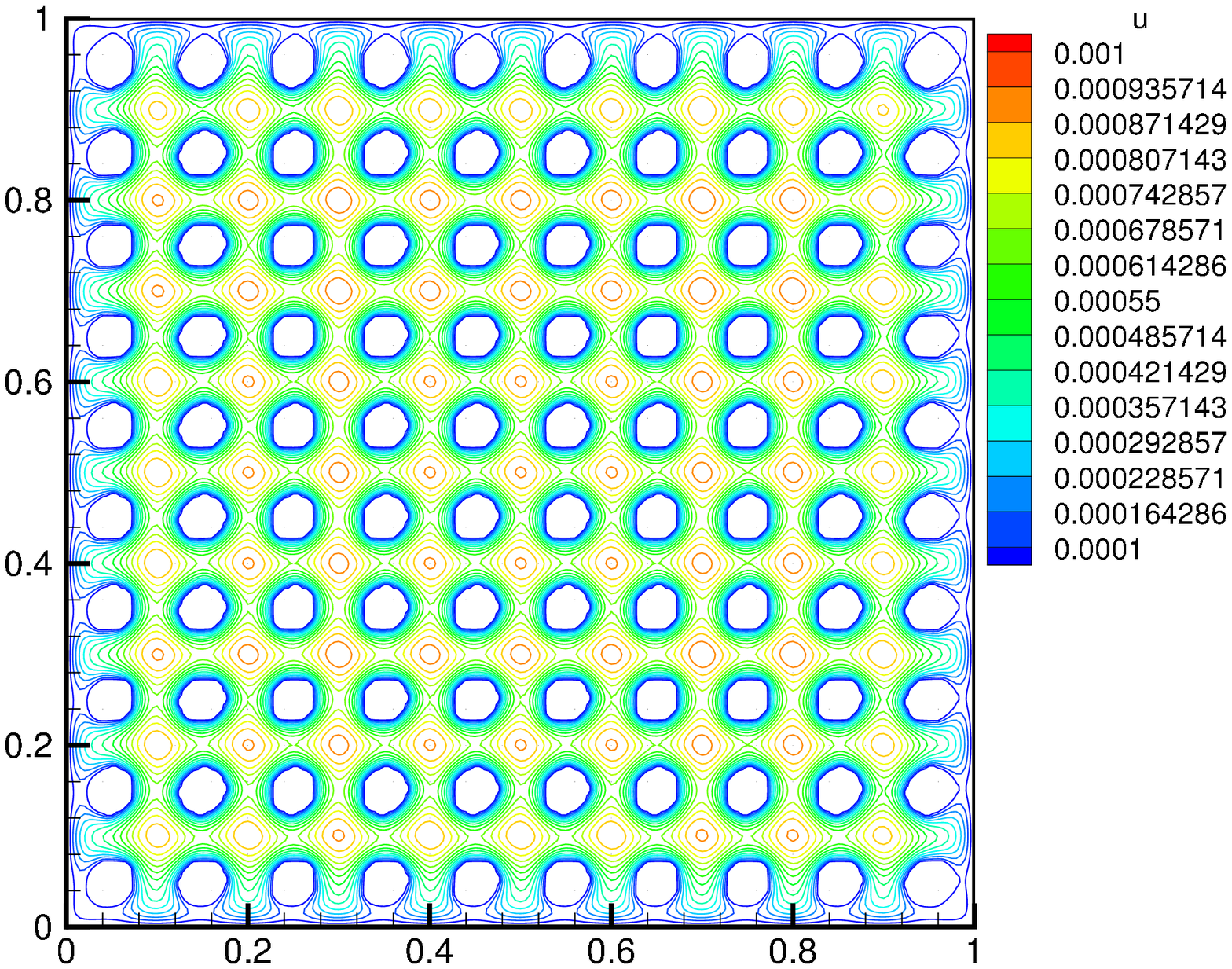}
\includegraphics[width=6.7truecm]{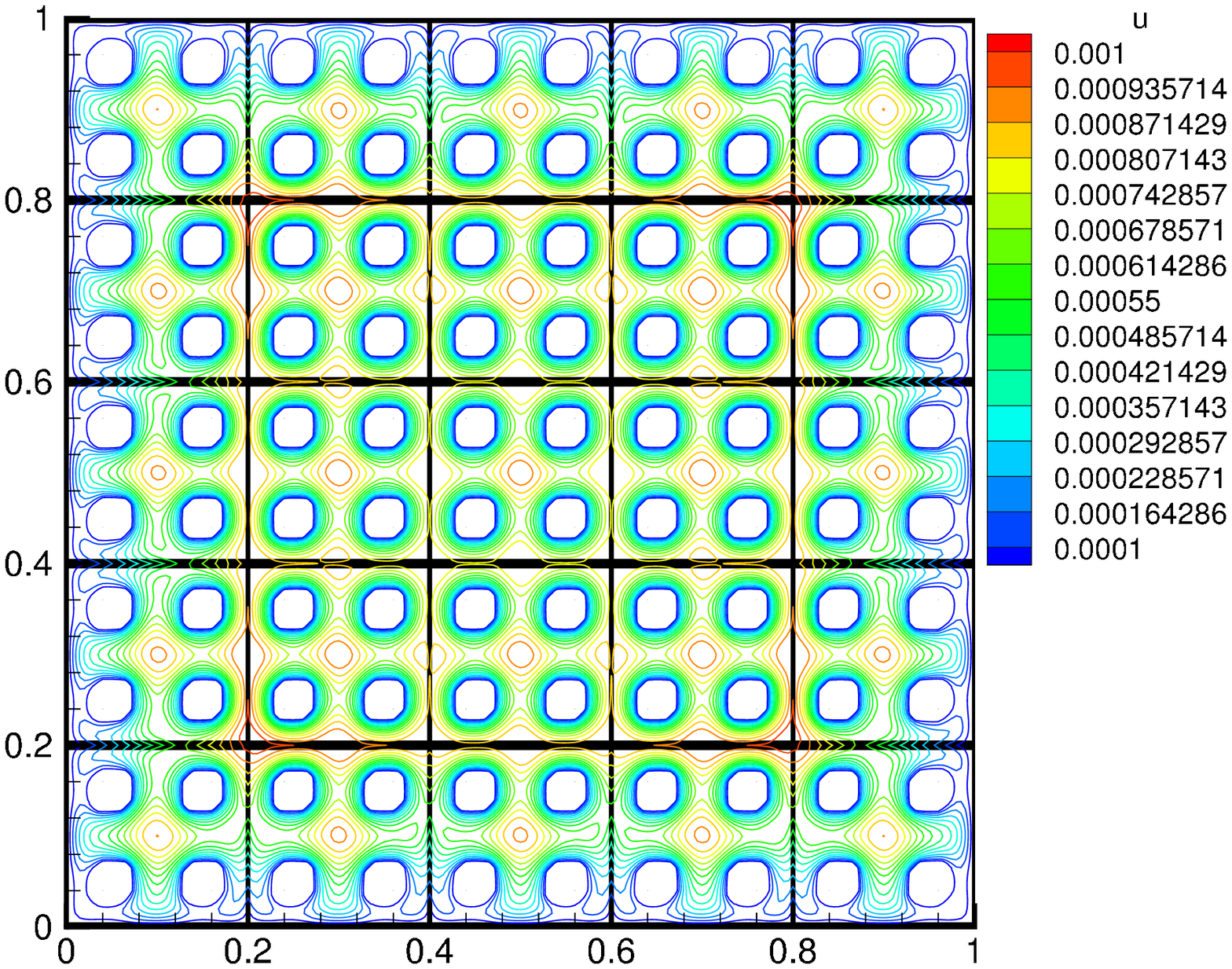}
\includegraphics[width=6.7truecm]{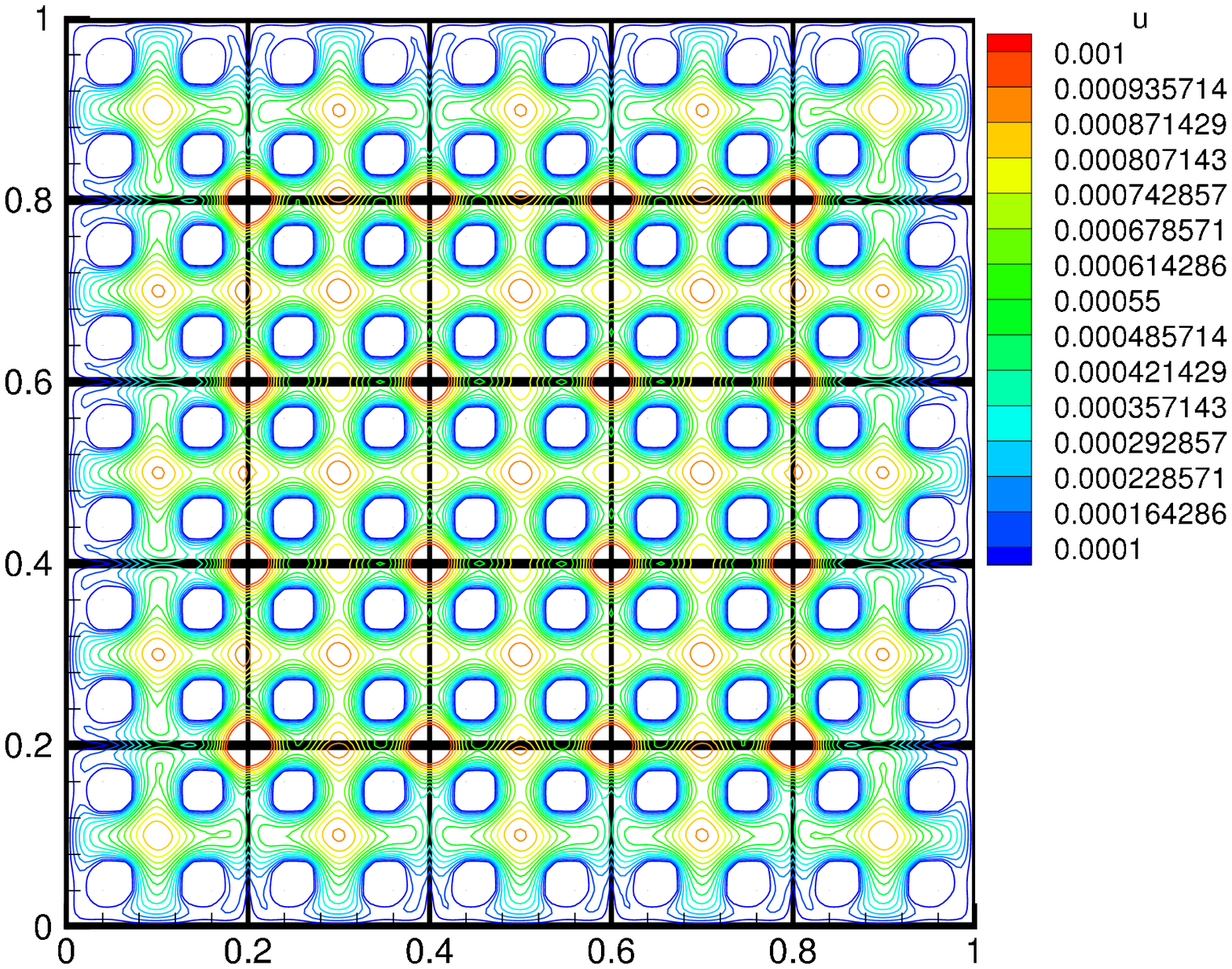}
\includegraphics[width=6.7truecm]{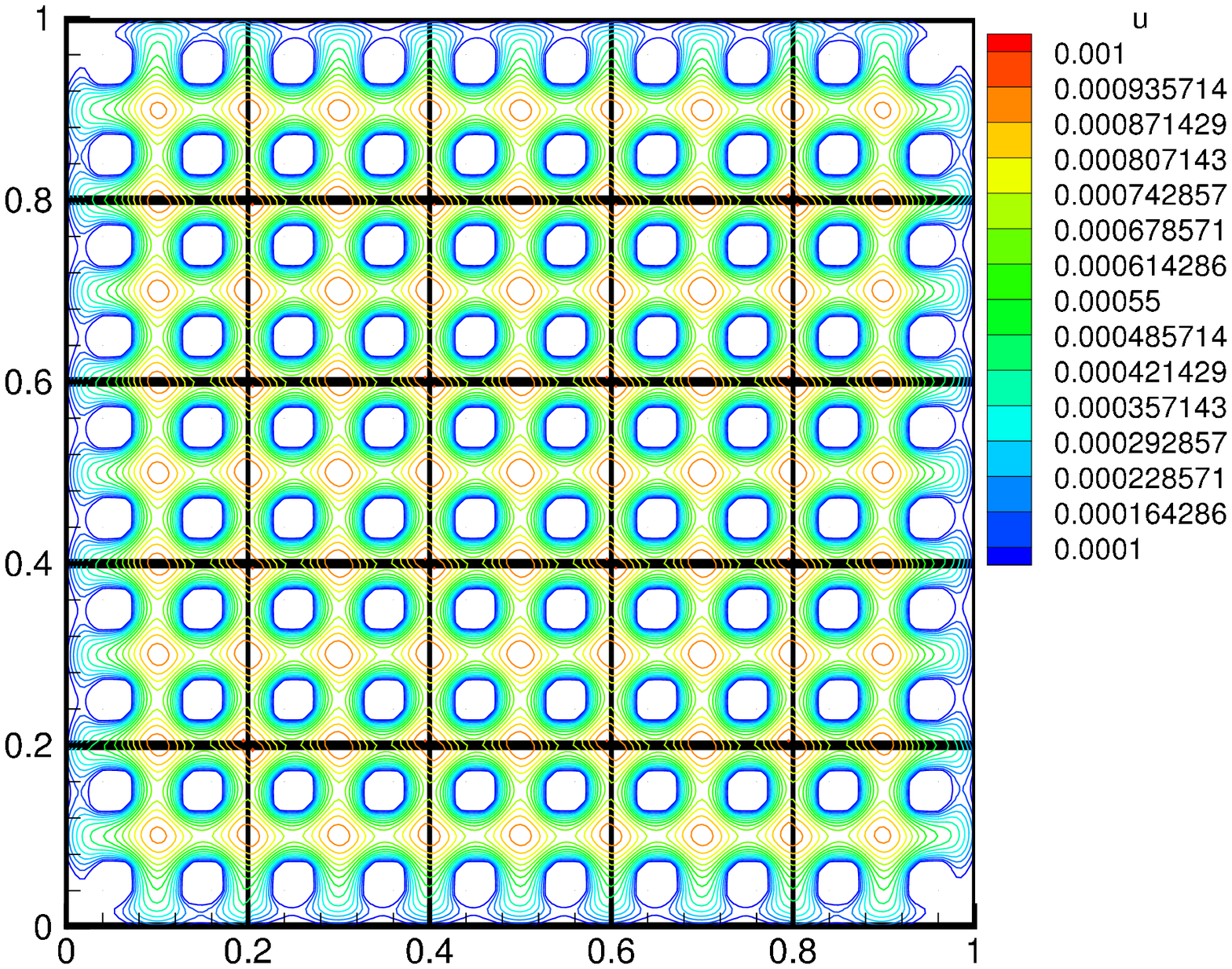}
\caption{(Test 1) Left to right and top to bottom: Reference solution
  (on the mesh $200 \times 200$), MsFEM with linear boundary conditions, MsFEM
  with oversampling (where the size of the quadrangles used to compute
  the basis functions is $3H \times 3H$), proposed MsFEM \`a la
  Crouzeix-Raviart. 
\label{fig:decalage-test1}}
\end{figure}

\begin{table}[h!]
\centering{
\begin{tabular}{l|c|c}
& $L^2$ error (\%) & $H^1$ error (\%) \\ \hline
MsFEM with linear conditions & 16 & 32 \\ \hline
MsFEM with oversampling & 20 & 38 \\ \hline
MsFEM \`a la Crouzeix-Raviart & 9 & 24%
\end{tabular}
}
\caption{Numerical relative errors for Test 1
\label{table:test1}}
\end{table}

\begin{figure}[htbp]
  \centering
\includegraphics[width=6.7truecm]{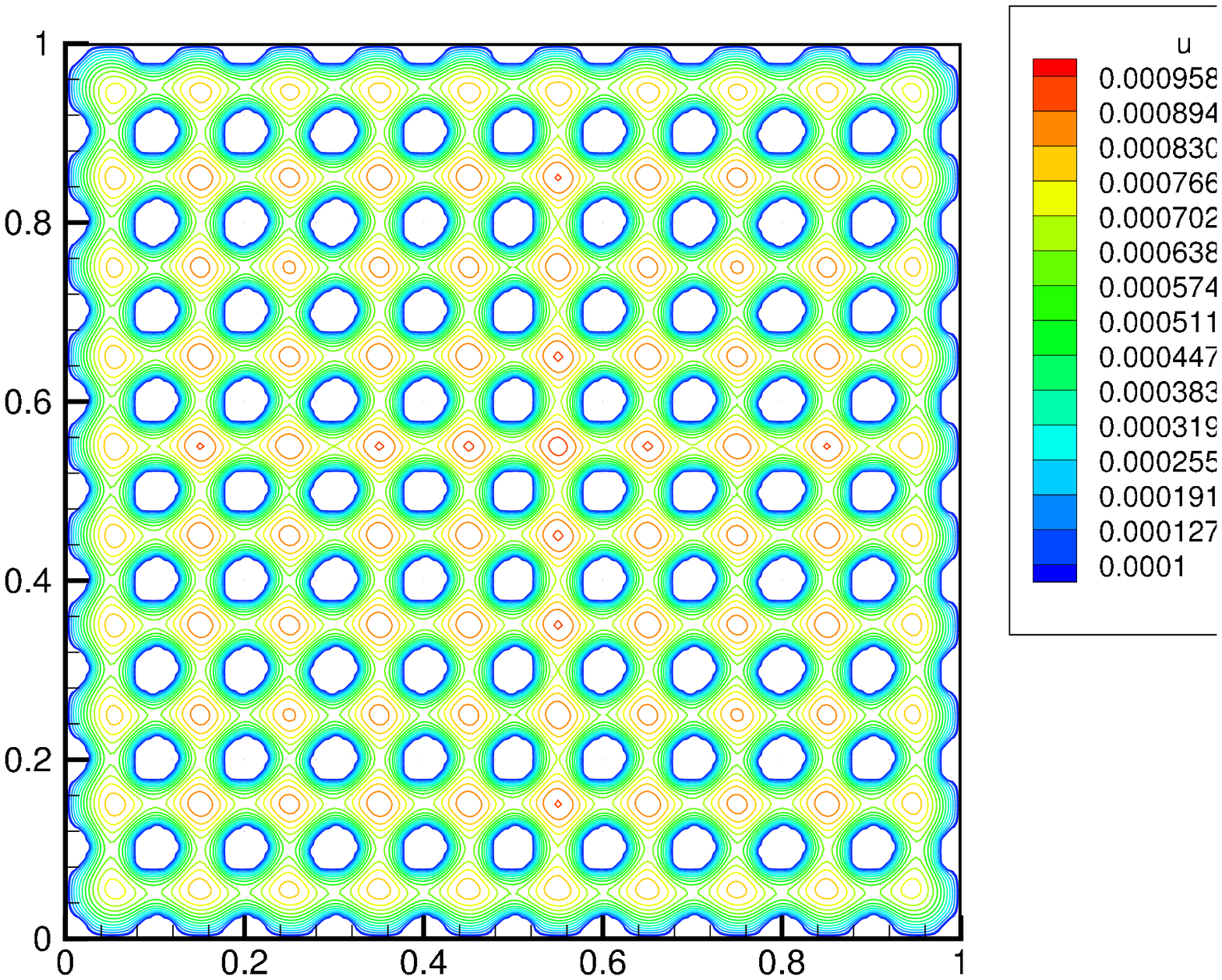}
\includegraphics[width=6.7truecm]{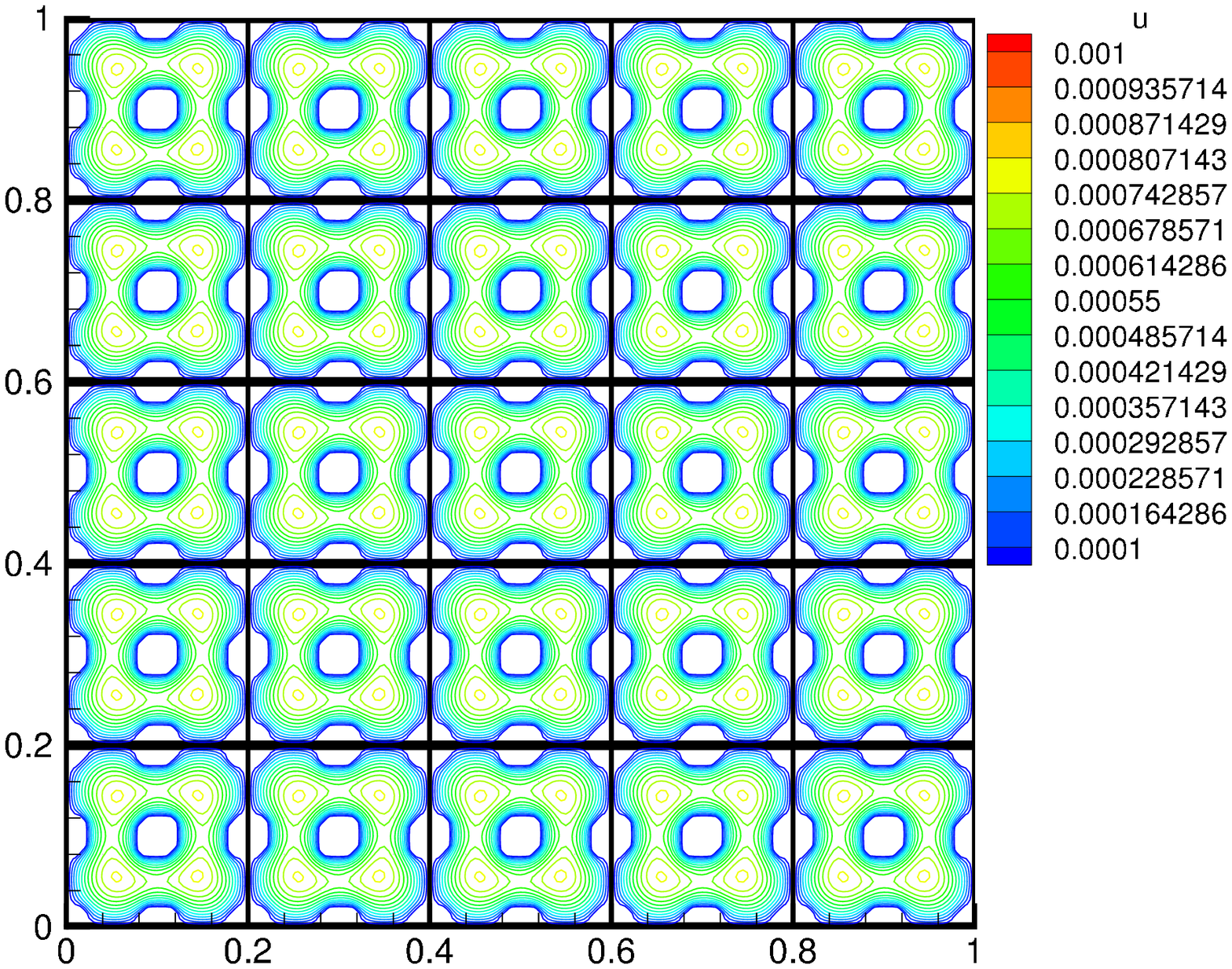}
\includegraphics[width=6.7truecm]{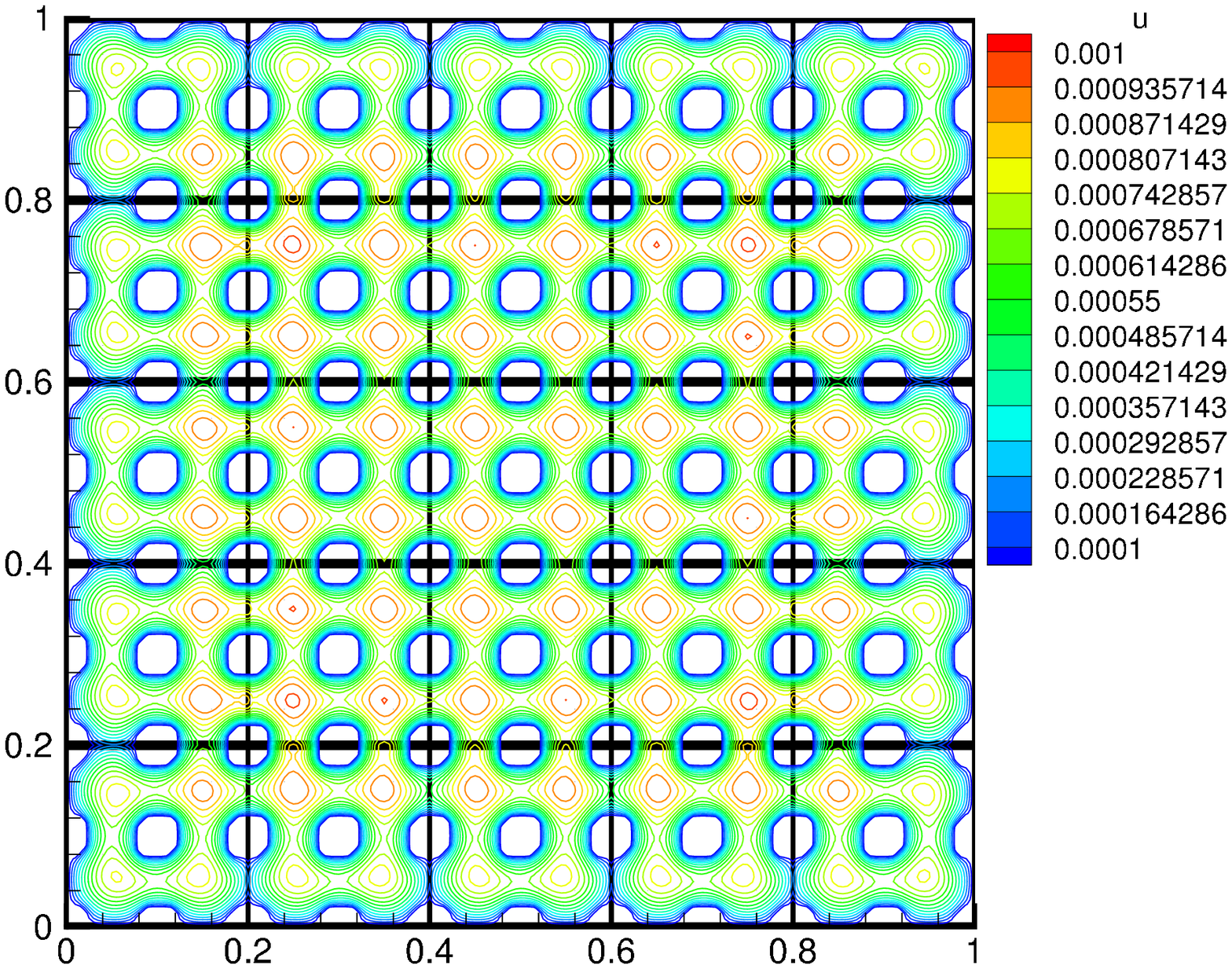}
\includegraphics[width=6.7truecm]{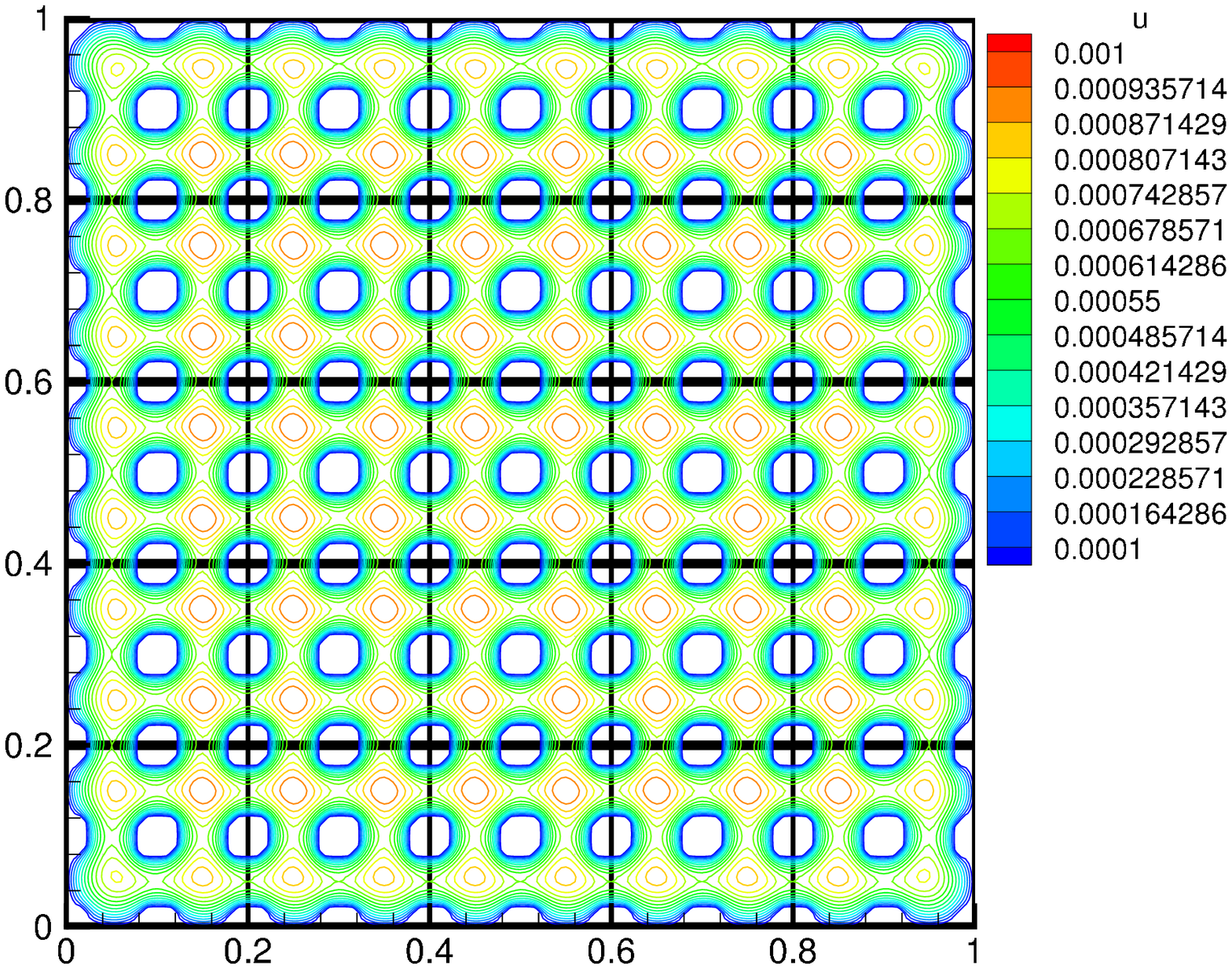}
\caption{(Test 2) Left to right and top to bottom: Reference solution
  (on the mesh $200 \times 200$), MsFEM with linear boundary conditions,
  MsFEM with oversampling (where the size of the quadrangles used to compute
  the basis functions is $3H \times 3H$), proposed MsFEM \`a la Crouzeix-Raviart.
\label{fig:decalage-test2}}
\end{figure}

\begin{table}[h!]
\centering{
\begin{tabular}{l|c|c}
& $L^2$ error (\%) & $H^1$ error (\%) \\ \hline
MsFEM with linear conditions & 28 & 52 \\ \hline
MsFEM with oversampling & 12 & 31 \\ \hline
MsFEM \`a la Crouzeix-Raviart & 9 & 27%
\end{tabular}}
\caption{Numerical relative errors for Test 2
\label{table:test2}}
\end{table}

\subsection{A test on a non-periodic geometry of perforations}
\label{sec:non-per}

A major motivation for using MsFEM approaches is to address
non-periodic cases, for which homogenization theory does not provide any
explicit approximation procedure. We have tested several such examples, two of them
being shown on Figure~\ref{fig:ex-non-per}. For each of them, the domain
$\Omega = (0,1)^2$ is meshed using quadrangles of size $H$, with $1/128
\leq H \leq 1/8$. The reference
solution is again computed on a mesh of size 
$1024 \times 1024$.

\begin{figure}[htbp]
  \centering
\includegraphics[width=6.7truecm]{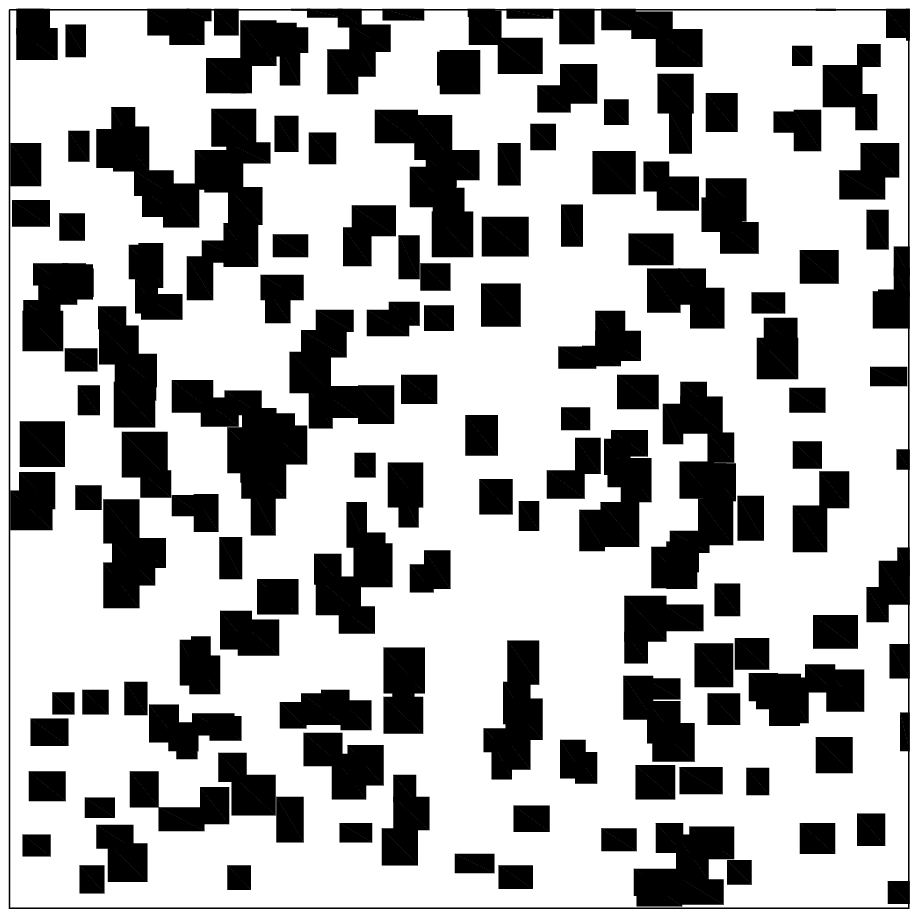}
\includegraphics[width=6.7truecm]{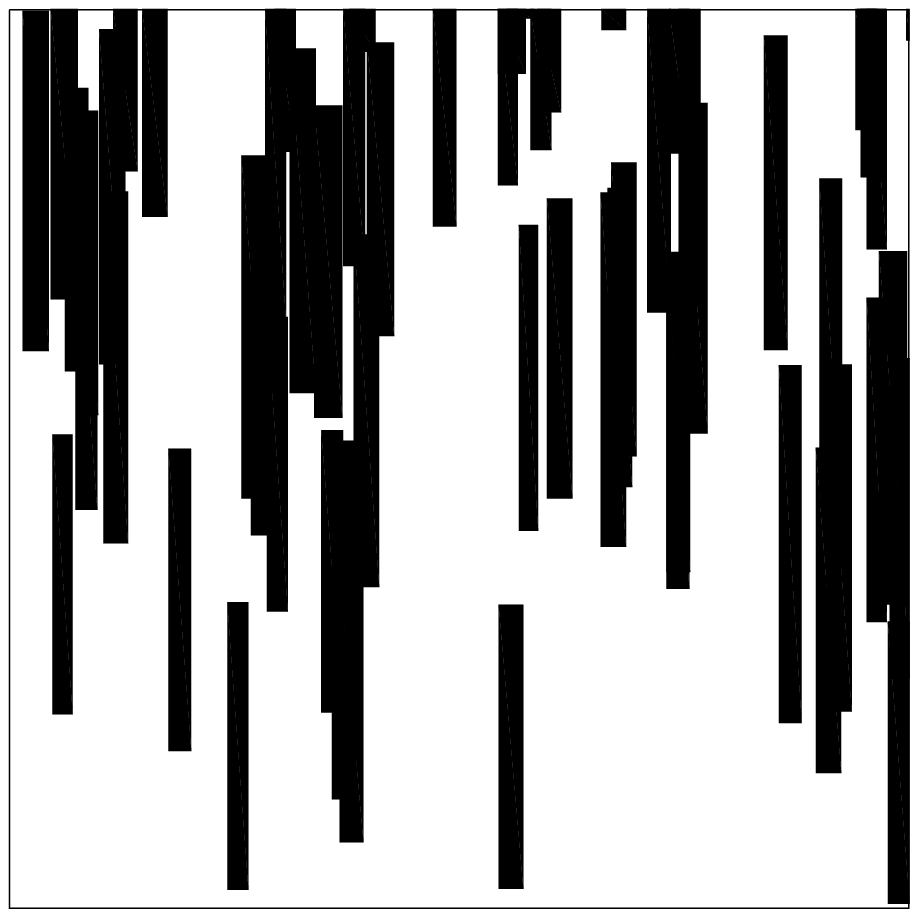}
\caption{Two examples of domains with non-periodic perforations (represented in black). Perforations have a rectangular shape, with a center randomly located in $\Omega=(0,1)^2$ according to the uniform distribution. Left: perforations are made from 100 rectangles, the width and height of which are uniformly distributed between 0.02 and 0.05. Right: perforations are made from 60 rectangles, the width (resp. the height) of which is uniformly distributed between 0.02 and 0.04 (resp. 0.02 and 0.4).
\label{fig:ex-non-per}}
\end{figure}

\medskip

Errors are shown on Figure~\ref{fig:errors-non-per3} (resp. Figure~\ref{fig:errors-non-per5}) for the test-case shown on the left (resp. right) part of Figure~\ref{fig:ex-non-per} (we have obtained similar results for several other test cases not shown here for the sake of brevity). We again see that 
our approach provides results at least as accurate as, and often more
accurate than the MsFEM approach with oversampling on quadrangles of
size $3H \times 3H$. Our approach outperforms all the other variants of
MsFEM that we have tested. These results confirm the definite interest
of the variant we introduce in this article.

\begin{figure}[htbp]
  \centering
\includegraphics[width=6.7truecm]{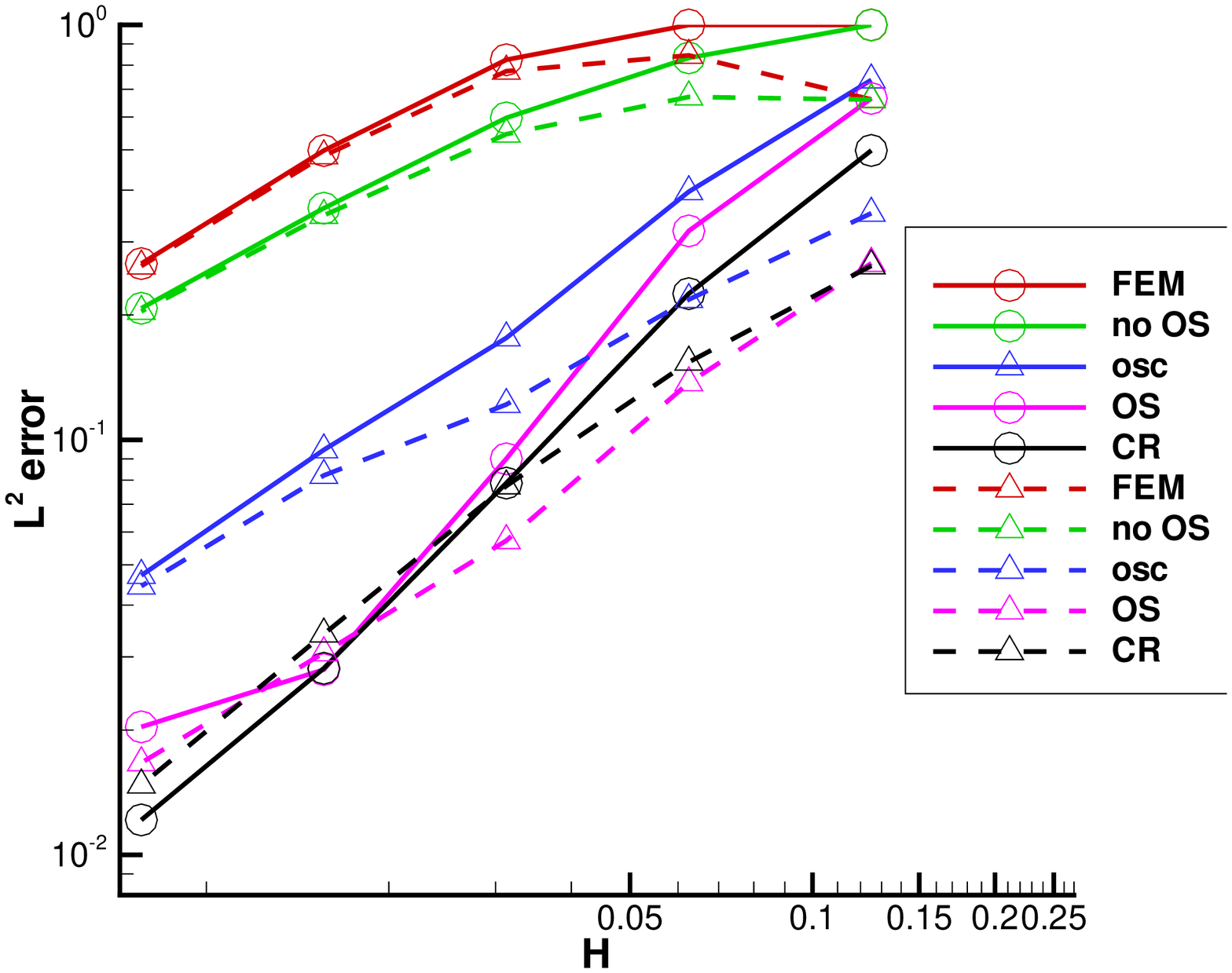}
\includegraphics[width=6.7truecm]{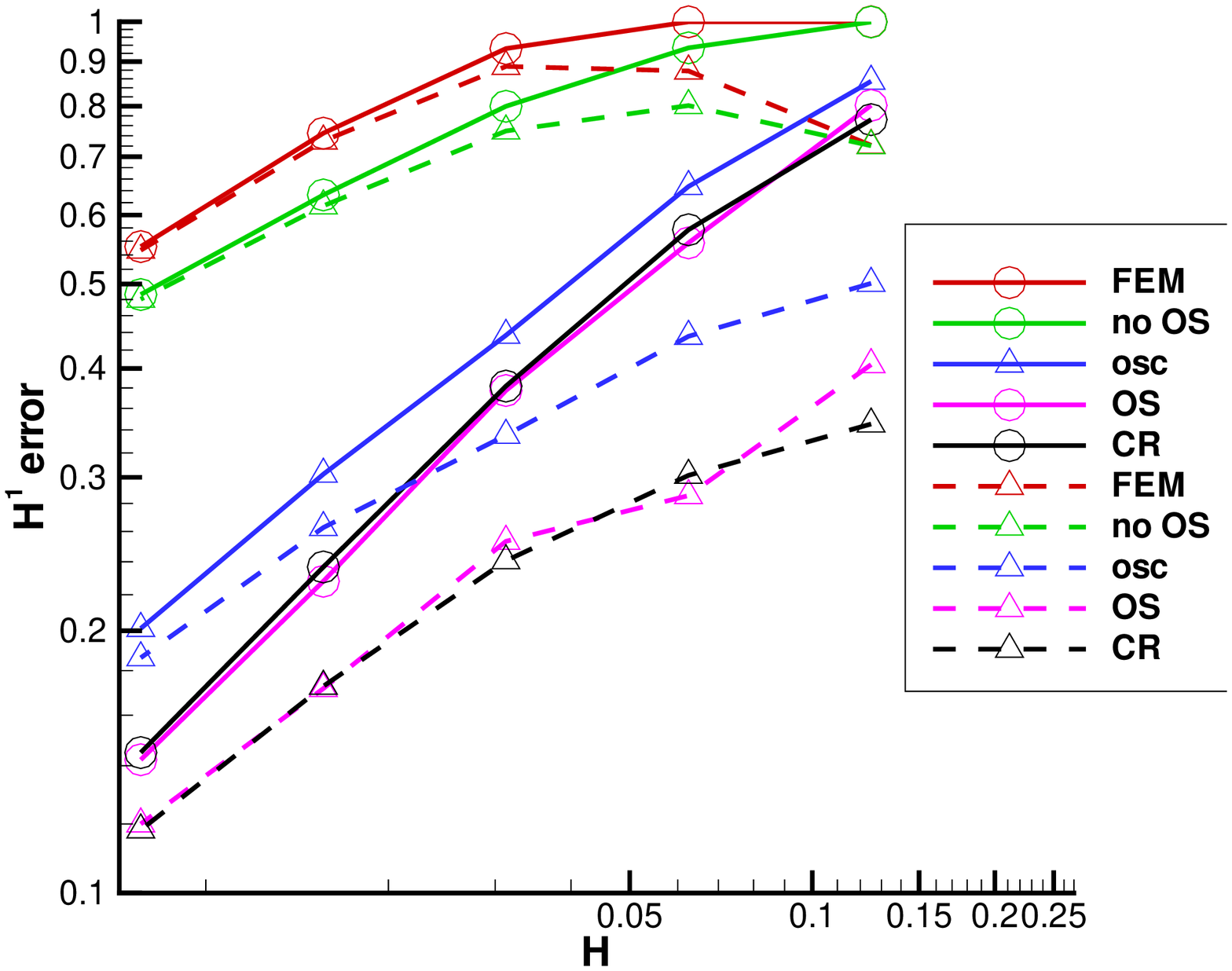}
\caption{Relative ($L^2$, left, and $H^1$-broken, right) errors with
  the same approaches as on Figure~\ref{fig:errors} for the test-case shown on the left-part of Figure~\ref{fig:ex-non-per}.
\label{fig:errors-non-per3}}
\end{figure}

\begin{figure}[htbp]
  \centering
\includegraphics[width=6.7truecm]{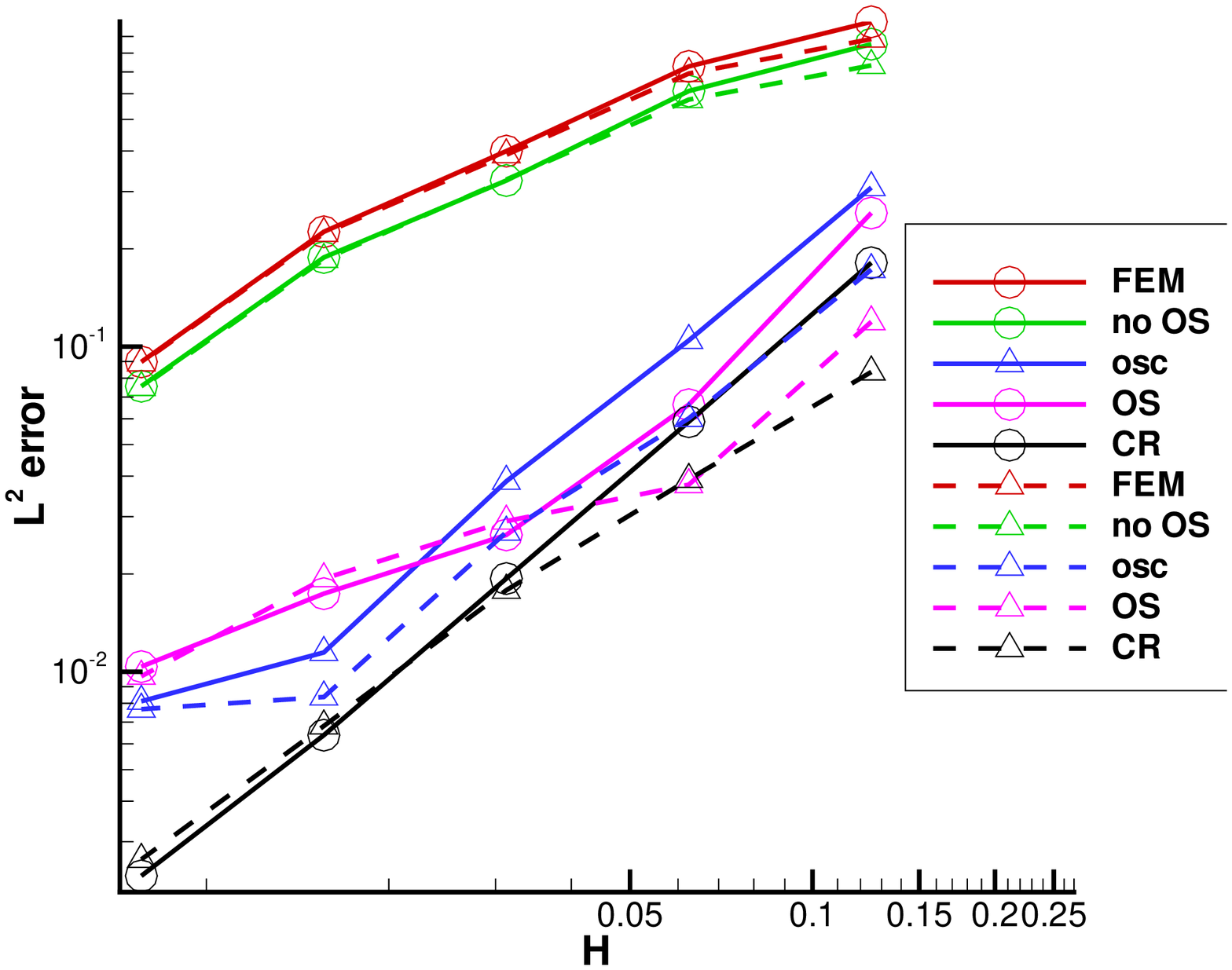}
\includegraphics[width=6.7truecm]{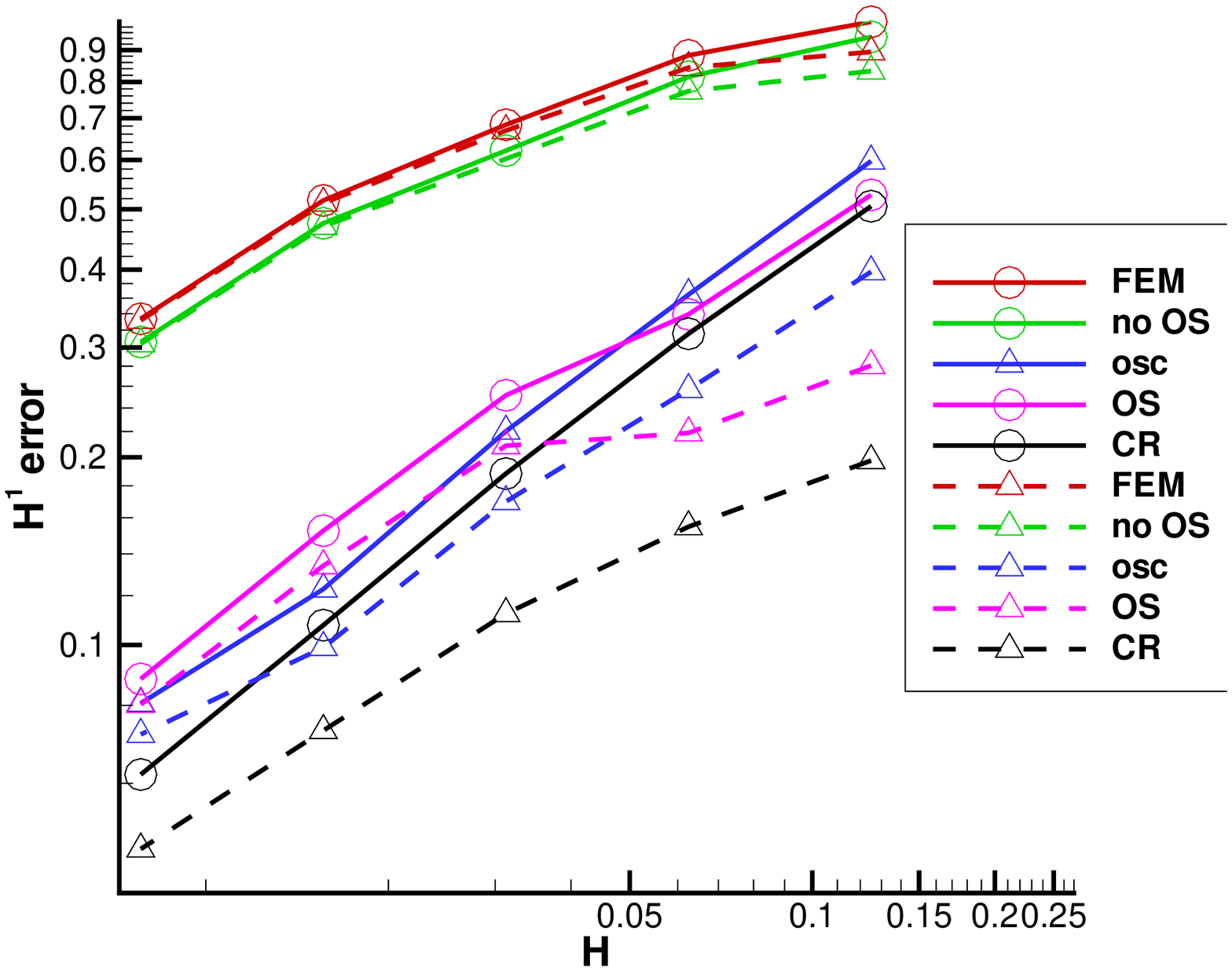}
\caption{Relative ($L^2$, left, and $H^1$-broken, right) errors with
  the same approaches as on Figure~\ref{fig:errors} for the test-case shown on the right-part of Figure~\ref{fig:ex-non-per}.
\label{fig:errors-non-per5}}
\end{figure}

\bigskip

\noindent{\bf Acknowledgments.} The work of the first two authors is
partially supported by ONR under Grant N00014-12-1-0383 and by EOARD under Grant FA8655-13-1-3061. The third author acknowledges the hospitality of INRIA. 
We thank William Minvielle for his remarks on a preliminary version of this article.

\appendix

\section{Technical proofs}

We collect in this Appendix the proof of two technical results used in Section~\ref{sec:proof}, namely the Poincar\'e inequality~\eqref{eq:poincare_perfore} and the homogenization result~\eqref{eq:lions-general}.

\subsection{The Poincar\'e inequality in perforated domains}
\label{sec:proof_poincare}

Consider the unit square $Y = (0,1)^d$ in dimension $d$, and some smooth perforation $B \subset Y$. There exists a constant $\mathcal{C}>0$ such that, for any $\phi \in H^1(Y \setminus B)$ with $\phi = 0$ on $\partial B$, we have
\begin{equation}
\label{eq:poinc_base}
\| \phi \|_{L^2(Y \setminus B)} \leq \mathcal{C} \| \nabla \phi \|_{L^2(Y \setminus B)}.
\end{equation}
Let $Y_k^B := k + (Y \setminus B)$ be the perforated unit cell after translation by the vector $k \in \ZZ^d$. We scale $Y_k^B$ by a factor $\eps$ and repeat this pattern periodically (with a period $\eps$ in all directions) for a finite number of times. We hence introduce 
\begin{equation}
\label{eq:def_K}
Q_\eps = \underset{k \in K}{\cup} \left( \eps Y_k^B \right),
\quad
K = \left\{ k \in \ZZ^d, \ a^-_i \leq k_i \leq a^+_i \ \text{for any $1 \leq i \leq d$} \right\}
\end{equation}
for some $a^-_i$ and $a^+_i$ in $\ZZ$, that we can also write as
$$
Q_\eps = R_\eps \setminus P_\eps,
$$ 
where $R_\eps$ is the quadrangle $R_\eps = \underset{k \in K}{\cup} \left( \eps (k+Y) \right)$ and $P_\eps$ is the set of perforations $P_\eps = \underset{k \in K}{\cup} \left( \eps (k+B) \right)$. Summing the inequality~\eqref{eq:poinc_base} for all cells and next scaling the geometry, we obtain that, for any $\phi \in H^1(R_\eps \setminus P_\eps)$ with $\phi = 0$ on $\partial P_\eps$, we have
\begin{equation}
\label{eq:poinc_base2}
\| \phi \|_{L^2(R_\eps \setminus P_\eps)} \leq \mathcal{C} \eps \| \nabla \phi \|_{L^2(R_\eps \setminus P_\eps)}
\end{equation}
where $\mathcal{C}$ is the same constant as in~\eqref{eq:poinc_base}.

Consider now $\phi \in H^1_0(\Omega_\eps)$. There exists a set $K$ of the form~\eqref{eq:def_K} such that $\Omega_\eps \subset Q_\eps$ (it is sufficient to include $\Omega_\eps$ into a sufficiently large perforated quadrangle). We now introduce $\overline{\phi}$, defined on $Q_\eps$ by
$$
\overline{\phi} = \phi \text{ on $\Omega_\eps$}, 
\quad
\overline{\phi} = 0 \text{ otherwise},
$$
and readily see that $\overline{\phi} \in H^1(Q_\eps)$ and $\overline{\phi} = 0$ on $\partial P_\eps$. The function $\overline{\phi}$ thus satisfies~\eqref{eq:poinc_base2}. We hence obtain
$$
\| \phi \|_{L^2(\Omega_\eps)}
=
\left\| \overline{\phi} \right\|_{L^2(R_\eps \setminus P_\eps)} 
\leq 
\mathcal{C} \eps \left\| \nabla \overline{\phi} \right\|_{L^2(R_\eps \setminus P_\eps)}
=
\mathcal{C} \eps \| \nabla \phi \|_{L^2(\Omega_\eps)}.
$$
This completes the proof of~\eqref{eq:poincare_perfore}.

\subsection{Homogenization result}
\label{sec:proof_hom}

In this section, we prove~\eqref{eq:lions-general}. To do so, we actually do not use~\eqref{eq:lions}. The proof below actually provides an alternative proof of~\eqref{eq:lions} (see Remark~\ref{rem:lions} below).

\smallskip

Let
$\eta^\varepsilon$ be a smooth function on $\overline{\Omega}$ that
vanishes on $\partial \Omega$, satisfies $0 \leq \eta^\varepsilon(x)
\leq 1$ on $\overline{\Omega}$ and is equal to 1 in 
$\omega_\varepsilon = \{ x\in \Omega \text{ s.t. } 
\mbox{dist}(x,\partial \Omega) > \varepsilon \}$. Using the fact that $\Omega$ is smooth, it is easy to see that
such a function can be constructed 
for each~$\varepsilon >0$ and we can suppose that it satisfies
\begin{multline}
\| \eta^\varepsilon \|_{L^\infty(\Omega)} \leq C,
\quad 
\| 1- \eta^\varepsilon \|_{L^2(\Omega)} \leq C \sqrt{\varepsilon},
\\
\left\| \nabla \eta^\varepsilon \right\|_{L^\infty(\Omega)} \leq 
\frac{C}{\varepsilon},
\quad 
\left\| \nabla \eta^\varepsilon \right\|_{L^2(\Omega)} \leq 
\frac{C}{\sqrt{\varepsilon}},
\quad 
\left\| \nabla^2 \eta^\varepsilon \right\|_{L^2(\Omega)} \leq 
\frac{C}{\varepsilon^{3/2}}
\label{eq:pty_eta}
\end{multline}
for some universal constant $C>0$. Set 
$\phi = u^\varepsilon -\varepsilon^2 \, w_\varepsilon \, f \, 
\eta^\varepsilon$, where $w_\eps(x) = w(x/\eps)$, with $w$ the solution to~\eqref{eq:corrector-lions}. We compute
\begin{eqnarray*}
-\Delta \phi  
&=&
f + \varepsilon^2 \Delta (w_\varepsilon \, f \, \eta^\varepsilon)
\\
&=&
f + (\Delta w)_\varepsilon \, f \, \eta^\varepsilon 
+ 2 \varepsilon (\nabla w)_\varepsilon \cdot \nabla (f \,
\eta^\varepsilon)
+
\varepsilon^2 w_\varepsilon \, \Delta (f \, \eta^\varepsilon) 
\\
&=&
f (1-\eta^\varepsilon)
+
2 \varepsilon (\nabla w)_\varepsilon \cdot \nabla (f \,
\eta^\varepsilon)
+
\varepsilon^2 w_\varepsilon \, \Delta (f \, \eta^\varepsilon)
\end{eqnarray*}
on $\Omega_\varepsilon$, where we have used~\eqref{eq:genP} in the first
line and the fact that $-\Delta w=1$ on $Y \setminus B$ in the last
line. Using the regularity~\eqref{eq:regul_w} of $w$ and the properties~\eqref{eq:pty_eta}
of $\eta^\varepsilon$, we deduce that
\begin{eqnarray}
&&
\| -\Delta \phi \|_{L^2(\Omega_\eps)}
\nonumber
\\
& \leq &
\| f \|_{L^\infty(\Omega)} \, \| 1- \eta^\varepsilon \|_{L^2(\Omega)} 
\nonumber
\\
&&+
2 \eps \| \nabla w \|_{L^\infty} \left( 
\| f \|_{L^\infty(\Omega)} \, \| \nabla \eta^\varepsilon \|_{L^2(\Omega)} 
+
\| \nabla f \|_{L^2(\Omega)} \, \| \eta^\varepsilon \|_{L^\infty(\Omega)} 
\right)
\nonumber
\\
&&+ 
\eps^2 \| w \|_{L^\infty} \Big( 
\| f \|_{L^\infty(\Omega)} \, \| \Delta \eta^\varepsilon \|_{L^2(\Omega)} 
+
2 \| \nabla f \|_{L^2(\Omega)} \, \| \nabla \eta^\varepsilon \|_{L^\infty(\Omega)} 
\nonumber
\\
&& \qquad \qquad \qquad \qquad +
\| \Delta f \|_{L^2(\Omega)} \, \| \eta^\varepsilon \|_{L^\infty(\Omega)} 
\Big)
\nonumber
\\
& \leq & 
C \sqrt{\varepsilon} \, {\cal N}(f),
\label{eq:jll}
\end{eqnarray}
where ${\cal N}(f)$ is defined by~\eqref{eq:def_N_f}.

We now notice that $u^\eps$ and $w_\eps \eta^\eps$ vanish on $\partial
\Omega_\eps$, hence $\phi = 0$ on $\partial \Omega_\eps$.
An integration by parts thus yields
\begin{equation}
\label{eq:avant}
\int_{\Omega_\varepsilon} |\nabla \phi |^2
=
\int_{\Omega_\varepsilon} (-\Delta \phi ) \, \phi 
\leq 
C \sqrt{\varepsilon} \, {\cal N}(f) \, 
\|\phi \|_{L^2(\Omega_\varepsilon)}.
\end{equation}
Inserting~\eqref{eq:poincare_perfore} in~\eqref{eq:avant}, we obtain
$| \phi |_{H^1(\Omega_\eps)} \leq C \varepsilon^{3/2} \, 
{\cal N}(f)$. We conclude by using the triangle inequality
$$
\left| u^\varepsilon - 
\varepsilon^2 w \left( \frac{\cdot}{\varepsilon} \right) 
f \right|_{H^1(\Omega_\eps)}
\leq
| \phi |_{H^1(\Omega_\eps)} + \eps^2 \left| 
w \left( \frac{\cdot}{\varepsilon} \right) f (1-\eta^\eps) 
\right|_{H^1(\Omega_\eps)},
$$
where both terms in the above right-hand side are bounded by 
$C \varepsilon^{3/2} \, {\cal N}(f)$. This yields the desired
bound~\eqref{eq:lions-general}. 

\begin{remark}
\label{rem:lions}
Note that if $f$ vanishes on $\partial \Omega$, we can take $\eta^\eps \equiv 1$ and~\eqref{eq:jll} is replaced by 
$$
\| -\Delta \phi \|_{L^2(\Omega_\eps)}
\leq
C \varepsilon \, {\cal N}(f).
$$
Following the same steps as above, we then recover the bound~\eqref{eq:lions}.
\end{remark}

\end{document}